\documentclass{amsart}
\usepackage{amsfonts,amssymb,graphicx}
\usepackage{epsfig,amsmath,latexsym}
\usepackage{amscd, amssymb, amsthm, amsmath,eucal, verbatim}
\usepackage{epstopdf}
\newcommand{\R}{{\mathbb{R}}}
\newcommand{\ZZ}{{\mathbb{Z}}}

\newtheorem{theorem}{Theorem}
\newtheorem{lemma}[theorem]{Lemma}

\newtheorem{corollary}[theorem]{Corollary}
 
\newtheorem{example}{Example}

\begin{document}
\title{Knotted Surfaces and Their Fold Curves}
\date{\today}
\author{Joel Hass}
\address{University of California, Davis, California 95616}
\email{jhass@ucdavis.edu}
\thanks{This work was started while the author was a Christensen Fellow at St. Catherine's College, Oxford.
Partially supported by NSF grant DMS:FRG 1760485 and BSF grant 2018313}
\begin{abstract}
This paper examines the relationship between the knotting of an embedded surface in $\R^3$ and the knotting of its fold curves, 
the curves forming the singular set of the projection of the surface  to a plane.
The first main result shows that every surface, no matter how knotted, can be isotoped so that its fold curves form an unlink.
A second result defines a new invariant which gives a complete obstruction to deforming a surface by an isotopy so as to turn a fixed curve on the surface into a fold curve.
\end{abstract}
\maketitle

\section{Fold Curves} \label{introd}

A projection $\pi \colon F \to \R^2$ of a smooth embedded surface in $\R^3$ to the $xy$-plane has singular points  where the derivative $d\pi$ is not injective.
These lie along  {\em fold curves}, along which
the tangent plane of the surface $f(F) \subset \R^3$  is perpendicular to the $xy$-plane.
The fold curves are visualized as the outline, or edges, of the surface when viewed from above.
For a generic surface (defined below) these fold curves are smoothly embedded in $\R^3$, while
their projections to the $xy$-plane are piecewise-smooth with a finite number of cusps. The torus   in Figure~\ref{profileseg} has two fold curves, one of which has a projection with four cusps.

 \begin{figure}[htbp] 
\centering
\includegraphics[width=2in]{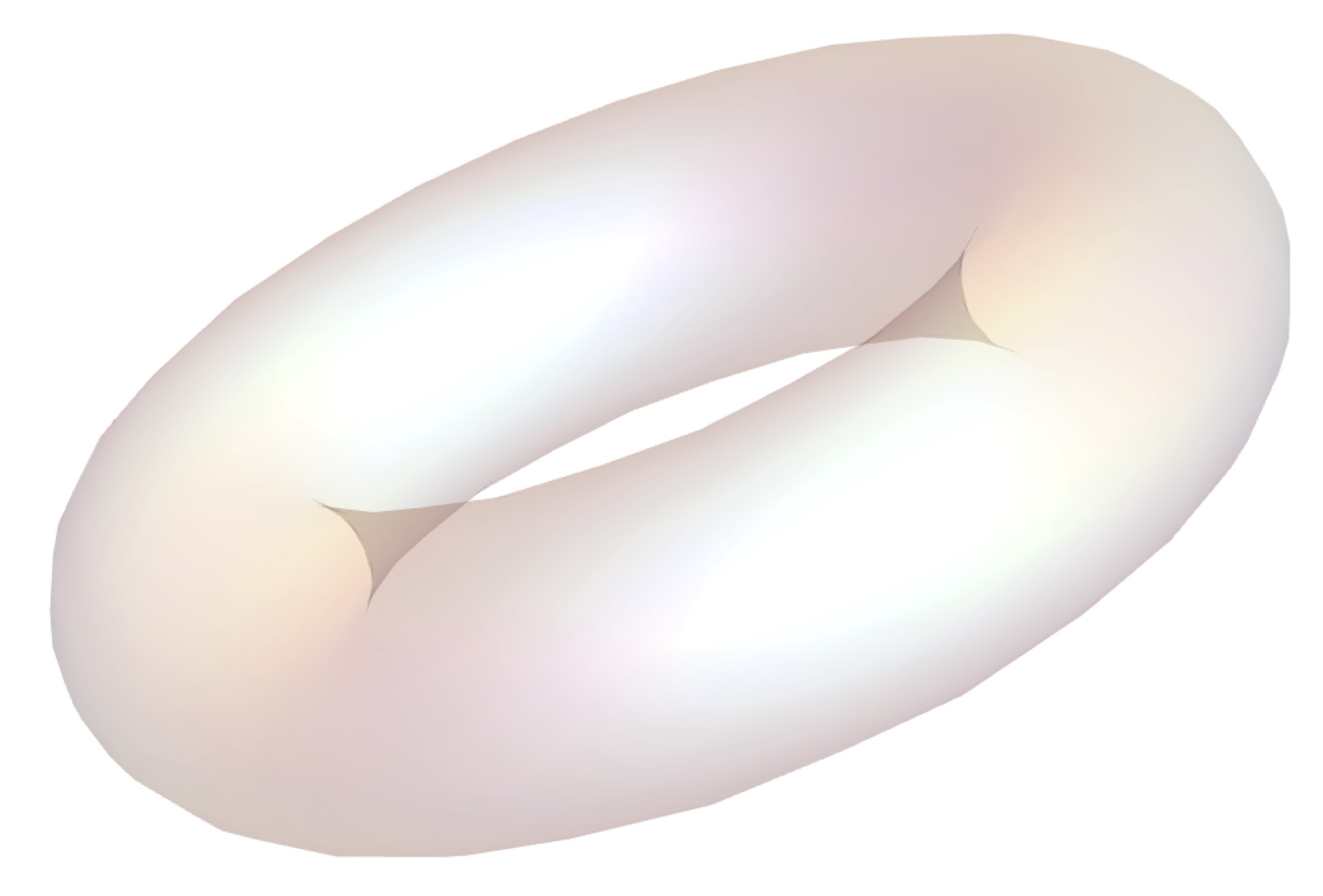} 
\caption{Singular points of planar projection form the fold curves of a surface embedded in $\R^3$.}
\label{profileseg}
\end{figure}

In this paper we present two main results concerning the fold curves.
In Theorem~\ref{unlink} we show that the link formed by the collection of fold curves of a surface can always be made trivial by an appropriate isotopy of the surface.
A surface of genus $\ge 1$, no matter how knotted, can be isotoped so that  its fold curves become an unlink with at most $g$ components.
The second main result, given in Theorem~\ref{delta}, specifies exactly when a curve $\gamma \subset  F$ can be turned into a fold curve.
We introduce an invariant  $\delta(\gamma, F)$ that vanishes exactly when it is possible for the curve  $\gamma$ to
be turned into a fold curve by an isotopy of $F$ that fixes $\gamma$.  
Moreover, we show that it is always possible to isotop a curve on a surface $F$ so that $\delta(\gamma, F)$ becomes zero.
This means that for any embedded curve $\gamma$ on a surface $F \subset \R^3$, there is an isotopic curve $\gamma_1$ on  $F$ and
an isotopy of  $F$  that  fixes $\gamma_1$ and turns $\gamma_1$ into a fold curve.

We now review the setting for these results.

\subsection{Background and terminology}

Throughout the paper all maps, embeddings, and isotopies are smooth ($C^\infty$).
 Two embedded surfaces are isotopic if they have parameterizations that are homotopic through embeddings.
 A surface is unknotted if it is isotopic to the boundary of a standardly embedded handlebody, i.e.  a handlebody that is
 a neighborhood of a planar graph.
The {\em singular points} of a map $f$ from a surface to the plane are the points where $df$ is not injective.
A smooth map  $f \colon F \to M$ between smooth manifolds is {\em stable} if every sufficiently small smooth perturbation
$g$ of $f$  has the same singularity structure as $f$, up to diffeomorphisms of the source and target.
 That is, there exist diffeomorphisms
$j \colon F \to F$  and $k \colon   M  \to M$ 
such that
$ g = k \circ f \circ j$.
 A  reference for the theory of stable maps is \cite{GG}.

The local singularities of a smooth map $f$  from a surface $F$ to the plane were investigated  by Whitney \cite{Whitney}.
He showed that for an open dense subset of $C^\infty(F, \R^2)$, which he called {\em excellent},  the singular points
consist of folds and cusps.  
Folds are locally modeled by the map $(u,v) \to (u,   v^2) $  and cusps by
$(u,v) \to (u, uv +  v^3)$, up to a smooth change of coordinates.  
Haefliger later considered the  nature of singularities of maps that arise by following  an embedding $f \colon  F \to  \R^3$
by an orthogonal projection $p(x,y,z) = (x,y)$  to the $xy$-plane \cite{Haefliger}.   Expositions can be found in \cite{Francis, Zeeman}.

These results established that an open and dense subset of the smooth embeddings of a closed surface $F$ into $\R^3$ have the property that
\begin{enumerate}
\item Following the embedding $f \colon  F \to  \R^3$  by projection $\pi$ to the $xy$-plane results in a  stable map  $\pi \circ f$ from the surface to $\R^2$.
\item The fold curves of a generic embedding form a collection of smooth regular curves in $\R^3$.
The projection of the fold curves to the plane gives a piecewise-smooth curve in general position that contains a finite number of double points, called nodes, and a finite number of cusps.
\end{enumerate}
We call such surface embeddings {\em generic}.

The collection of piecewise-smooth curves formed by projecting the fold curves to the plane is called the {\em apparent contour}.
It is smooth except at a finite number of cusps whose
preimages on the fold curves are called {\em cuspidal points}.  At a cuspidal point the fold
curve has a vertical tangency.
While an apparent contour curve is only piecewise-smooth, each fold curve is a smoothly embedded regular curve in $\R^3$, and the collection of
fold curves defines a link.
Up to a smooth change of coordinates, a local model for a surface near a cuspidal point $f(0,0)$ is given by 
the parameterization
$$
f(u,v) = (u,  v^3-uv, v),
$$
shown in Figure~\ref{fig:cusp}.   The fold curve in these coordinates is given by
$\gamma(t) = (3t^2, -2t^3,  t)$.  It has a cuspidal point at $\gamma(0) = (0,0,0)$ where the tangent vector  $<0,0,1>$ is vertical.
The apparent contour is the planar curve given by  $\pi \circ \gamma(t) = (3t^2, -2t^3)$ and has a cusp at $\pi \circ \gamma(0) = (0,0)$.
The precise form of this equation will not be needed in our applications.
The apparent contour in Figure~\ref{fig:cusp} has a {\em left cusp}, in which the curve turns to the left at the cusp. 
Reflecting through the $xz$-plane gives a similar {\em right cusp}.

\begin{figure}[htbp]
    \centering
    \begin{minipage}{0.7\textwidth}
        \centering
        \includegraphics[width=\linewidth]{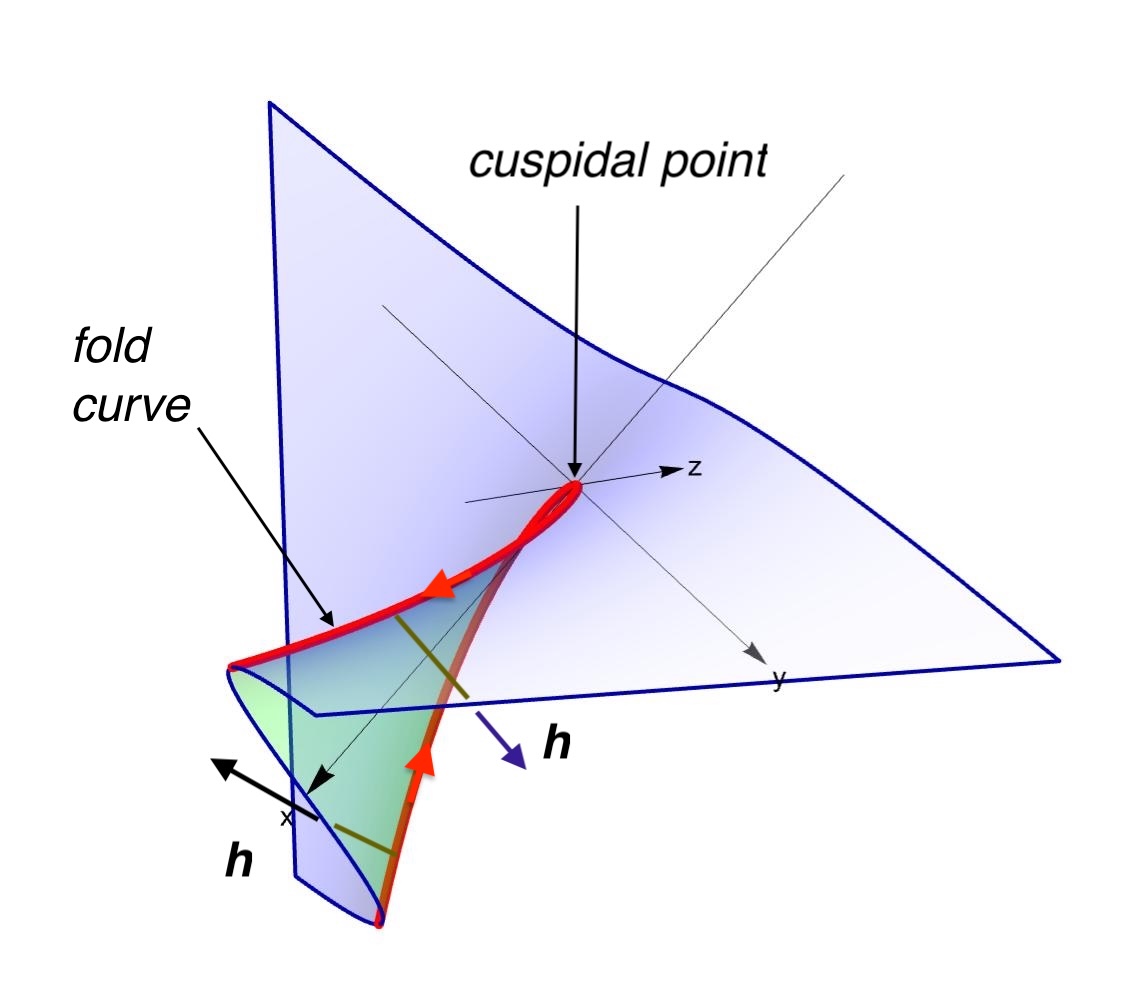}
    \end{minipage}
    \hfill
    \begin{minipage}{0.25\textwidth}
        \centering
        \includegraphics[width=\linewidth]{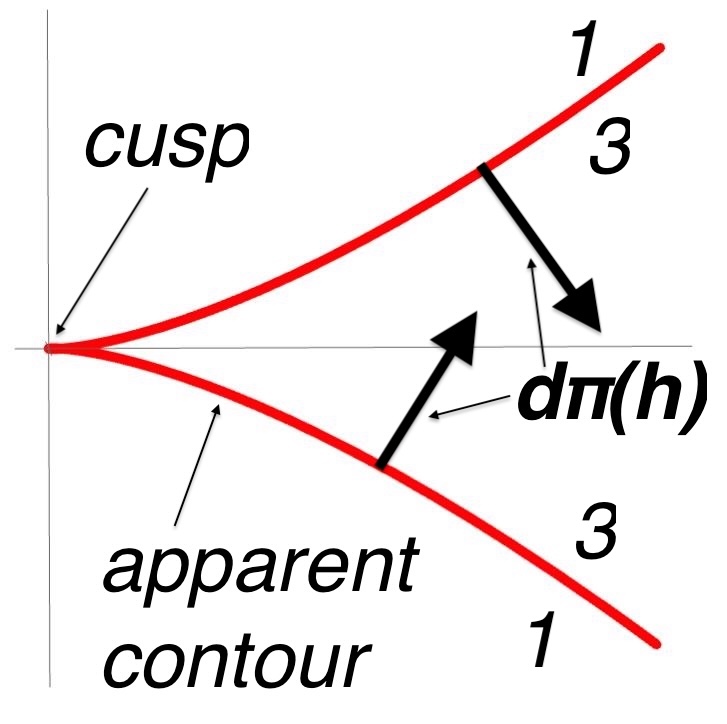}
            \end{minipage}
\caption{A projection of a generic surface has singularities along fold curves, shown at left.
The projection of these curves to the plane forms the apparent contour. A cuspidal point occurs where the fold curve is vertical and projects to a cusp on the apparent contour, shown at right. The number of sheets above each region in the plane is indicated.}
\label{fig:cusp}
\end{figure}
 

Away from the cuspidal points, a  vertical line near a fold curve $\gamma$ intersects the surface in two more points on one side of $\gamma$  than on the other.
We take  $\mathbf{h}$   to be the horizontal unit vector that is normal to the surface along the fold curve and points to this two-sheeted side. 
The vector $\mathbf{h}$ flips from one side of the surface to the other as $\gamma$  crosses a cuspidal point.
Away from cuspidal points $\mathbf{h}$ induces an orientation on the fold curve 
in which the pair of vectors $( \mathbf{\dot \gamma}, \mathbf{h} )$  agrees with the orientation of the $xy$-plane, after projection.
This orientation of $\gamma$  is consistent as a fold curve passes through a cuspidal point, and thus gives a global orientation to the fold curve. 
A closed surface $F \subset \R^3$ is oriented, and  has an outward pointing unit normal vector field $\mathbf{n_F}$, which coincides with $\pm \mathbf{h}$  along a fold curve
away from cuspidal points.
The  inner product of $\mathbf{h}$ with $\mathbf{n_F}$ along a fold curve is positive on one side of a cuspidal point and negative  on the other.

\subsection{Motivation and Applications}
Fold curves play an important role in computer graphics, imaging, visualization, and reconstruction of objects in  three dimensions. 
In different settings fold curves are called the {\em contour generators},  {\em profile curves}, 
{\em critical sets}, {\em rims}, or {\em outlines} of a surface. 
The apparent contour is also referred to as the {\em branch set},
 {\em occluding contour}, {\em silhouette} or {\em extremal boundary} of a surface projection
\cite{GG, BoyerBerger, CipollaBlake, FukudaYamamoto, Giblin, Pignoni, PlantingaVegter, VaillantFaugeras}. 
 In images of transparent objects, fold curves can remain visible while the rest of the 
 surface is barely seen. 
Fold curves also appear in the visualization of an embedded surface in $\R^4$, which can be described by first projecting to $\R^3$ and then to the plane \cite{Carter}.
The problem of  constructing surfaces in $\R^3$ that have given apparent contours is discussed in \cite{Bellettini, Hacon, Hacon2, MenascoNichols}.

While fold curves have been extensively studied, the author is not aware of previous results connecting their knot and link types to the isotopy class of the surface on which they lie.

\subsection*{Outline} 
In Section~\ref{knotted}  we examine how the isotopy class of a surface determines the knotting of its fold curves.
It seems reasonable to guess that a highly knotted surface must necessarily contain at least one non-trivial knot among its collection of fold curves,
as suggested by the knotted torus in Figure~\ref{trefoils} with its two knotted fold curves.

\begin{figure}[htbp] 
\centering
 \includegraphics[width=1.5in]{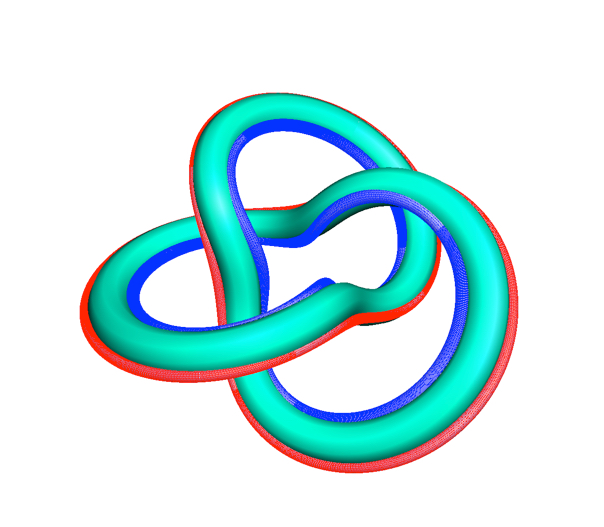}  \includegraphics[width=1.5in]{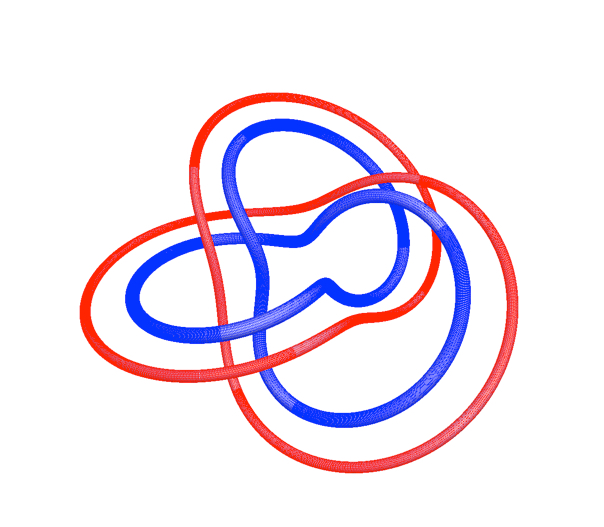} 
\caption{A knotted torus in $\R^3$ that has two fold curves, each isotopic to a trefoil knot.}
\label{trefoils}
\end{figure}

However, this is not the case. There is an isotopic torus whose full set of profile curves consists of a single unknotted loop.
Theorem~\ref{unlink} shows that any
 embedded genus-$g$ surface $F \subset \R^3$, no matter how knotted or convoluted, can be isotoped 
 so that its  collection of fold curves becomes an unlink with at most $g$ components. 
In many cases the entire collection of fold curves can be reduced to a single unknotted curve.  
In particular, this can be done for any genus-one surface, and for any surface that bounds a handlebody or punctured handlebody in $\R^3$.
 
In Section~\ref{CurvetoFold} we investigate the situation where the surface, rather than the fold curve, is unknotted.
Is it possible for an unknotted surface to generate a knotted fold curve? 
We show that topologically this is possible in great generality.
Any knot that embeds on a surface of genus $g$ can be realized as a fold curve after a suitable isotopy of the curve and surface.
For example, any torus knot can be realized as a fold curve of an unknotted torus, and any pretzel knot as a fold curve of an unknotted genus-two surface.  Note that we are not claiming that these knots form the full set of fold curves, but rather that they constitute one component of this set.
Corollary~\ref{afterisotopy} shows that given a surface $F$ embedded in $\R^3$ and a curve $\gamma_0 \subset F$, 
there is no obstruction to finding an isotopic curve $\gamma_1 \subset F$ and then realizing $\gamma_1$ as a fold curve following an isotopy of $F$ in $\R^3$.

The situation becomes more restricted if we try to find an isotopy of $F$ that turns a curve $\gamma$ into a fold curve while fixing the curve.
We define an  invariant \emph{$\delta(\gamma, F)$} that gives an obstruction
to achieving this. 
The $\delta$ invariant is analogous to the Thurston--Bennequin number of a Legendrian knot, but
is defined in the setting of general smooth curves and surfaces.

In Section~\ref{restrictions} we give some geometric applications of these results,
including examples of planar curves that cannot arise from projections of a fold curve on prescribed surfaces.  We conclude with some open questions.


\section{Fold curves of Knotted Surfaces} \label{knotted}

Our first result shows that a surface $F_0$, no matter how knotted,  can be isotoped to a surface $F_1$ whose entire collection of fold curves 
consists of an unlink. For a sphere this unlink has one component, while for a surface of genus $g \ge 1$ it can be chosen to have at most $g$ components.

\begin{theorem} \label{unlink}
An embedded connected surface $F \subset \R^3$ of genus $g \ge 1$ can be isotoped so that its full collection of fold curves consists of an unlink with at most $g$ components. 
\end{theorem}
\begin{proof} 

We use the following property of surfaces in $\R^3$.
\begin{lemma} \label{tubes}
A genus-$g$ surface $F \subset \R^3$, with $g \ge 1$,
is isotopic to a surface obtained by starting with $k$ disjoint 2-spheres
and successively adding $(g+k-1)$ 1-handles, where $1 \le k \le g$.
\end{lemma}
\begin{proof} 
The Loop Theorem \cite{Papakyriakopoulos} implies that any closed surface  $F \subset \R^3$ other than a 2-sphere is compressible. 
This means that there is an essential compressing disk for the surface, namely a disk properly embedded in the complement of $F$ whose boundary curve
does not bound a disk on $F$.
Surgering the surface along a compressing disk, or compressing along the disk,  results in a surface of lower genus if the boundary of the disk is non-separating on $F$. 
If the boundary of the disk separates $F$, compressing along the disk results in two surfaces, each of lower and positive genus, and with the same summed genus.
Repeating this process
eventually results in a collection of $k$ disjoint 2-spheres.  
If $k > g$ then at least one compression splits off a 2-sphere, and thus does not result from compressing along an essential disk.  
Since all compressions are carried out along essential disks, we have $1 \le k \le g$. 
Reversing this process, we can reconstruct $F$ by successively adding $(g+k-1)$ 1-handles to the $k$ 2-spheres.  
When a 1-handle connects two different components it does not change the total summed genus, and when it runs from a component to itself it adds one to the total genus. 
\end{proof} 

Note that it is not always possible to take $k=1$, as shown by the genus-two surface in Figure~\ref{k=2}.  A standard cut and paste argument shows that any compression of this surface separates it into two components.

\begin{figure}[htbp] 
\centering
\includegraphics[width=3in]{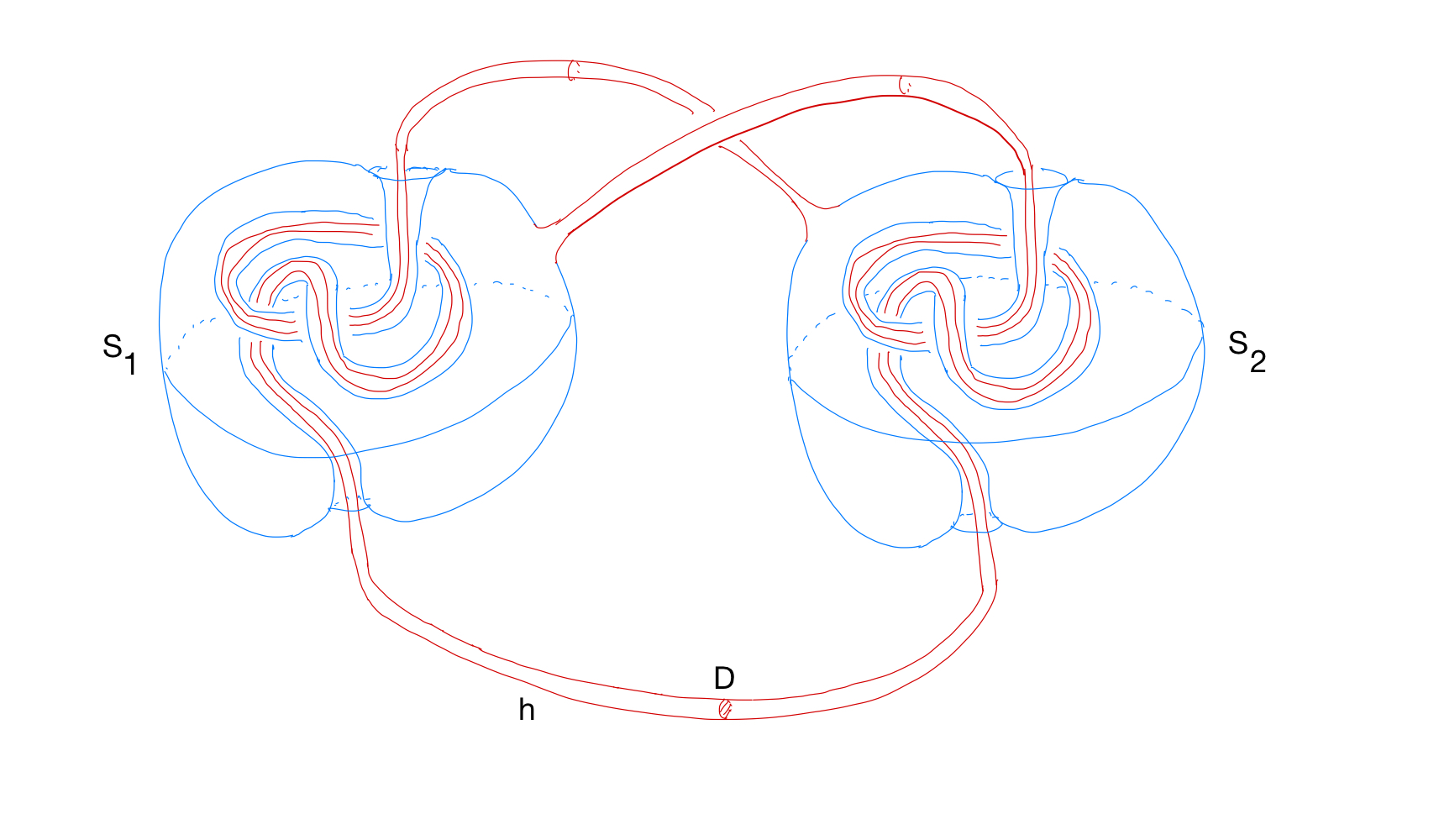}  
\caption{A genus-two surface for which any essential compressing disk has a boundary that separates the surface. Compressing the 1-handle $h$ along the disk $D$ results in two tori, which in turn compress to two 2-spheres.}
\label{k=2}
\end{figure}

\subsection{Fold-curve band sum}
We now describe how to construct a surface whose fold curves are compatible with the handle structure described in 
Lemma~\ref{tubes}.  

A {\em fold-curve band sum} adds a 1-handle to $F$ while 
changing its fold curves by a band sum. See Figure~\ref{fig:bandsum}. 

\begin{figure}[htbp] 
\centering
\includegraphics[width=2in]{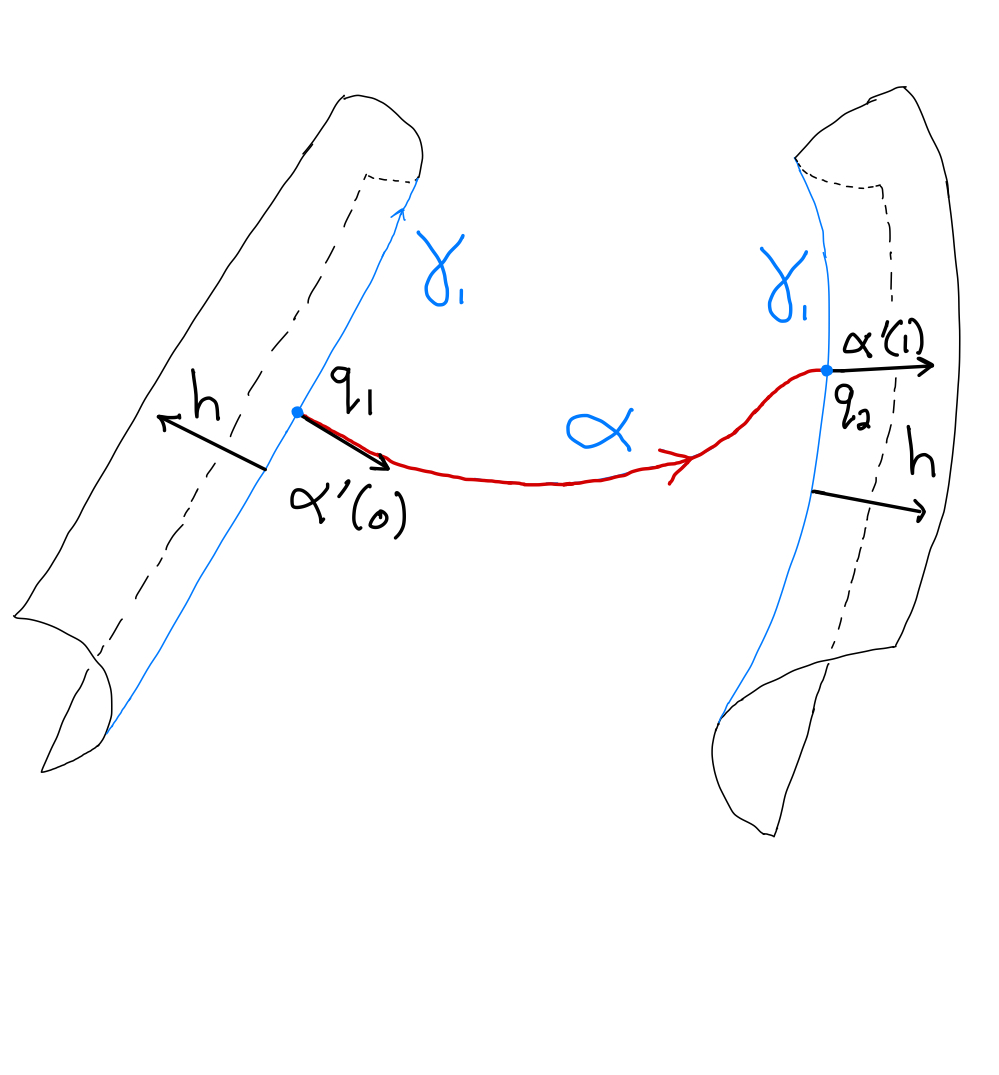}  \includegraphics[width=2in]{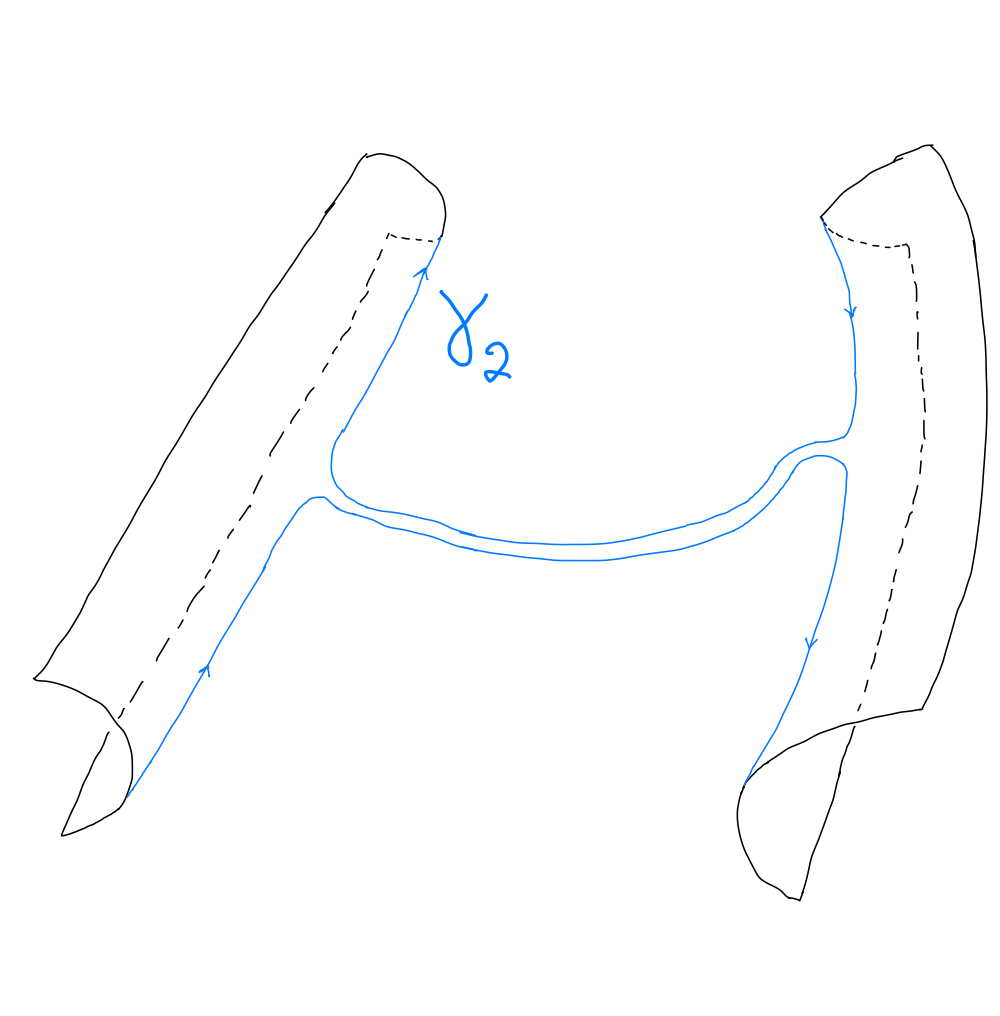}  
\caption{A fold-curve band sum  along an arc $\alpha$ adds a 1-handle to a surface and performs a band sum on its fold curves.}
\label{fig:bandsum}
\end{figure}

We specify this construction using the unit normal vector $\mathbf{h}$  associated to the fold curve.
Suppose that two points $q_1, q_2 $ each lie on fold curves of $F$, 
and that a smooth embedded arc $\alpha \subset \R^3$ runs from $q_1$ to $q_2$ in the complement of $F$.  Moreover, suppose that
$\mathbf{\alpha'}(t)$ is  never vertical, $\mathbf{\alpha'}(0) = -\mathbf{h} |_{q_1}$ and $\mathbf{\alpha'}(1) = \mathbf{h} |_{q_2}$.
For $\epsilon$ sufficiently small, an $\epsilon$-tubular neighborhood of $\alpha$ forms  a 1-handle that is attached to $F$ along disk neighborhoods of $q_1$ and $q_2$,
as shown in Figure~\ref{fig:bandsum}. 
The fold curves of the resulting surface are obtained as a band sum of the fold curves of the initial surface $F$.

Next we define two explicit local constructions on a surface that give useful fold curve structures on a surface.  
\subsection{Pockets}

The first construction is a local deformation that adds a small overhang to a surface near a point that is away from the fold curves.  The construction takes place 
inside an open disk that is disjoint from existing fold curves, and adds
an unknotted component to the collection of fold curves of the surface as shown in Figure~\ref{pocket}. We call this configuration a {\em pocket}. 
After the addition of a pocket, a 1-handle coming from a fold-curve band sum can be attached to either side of the surface.

\begin{figure}[htbp] 
\centering
\includegraphics[width=1.5in]{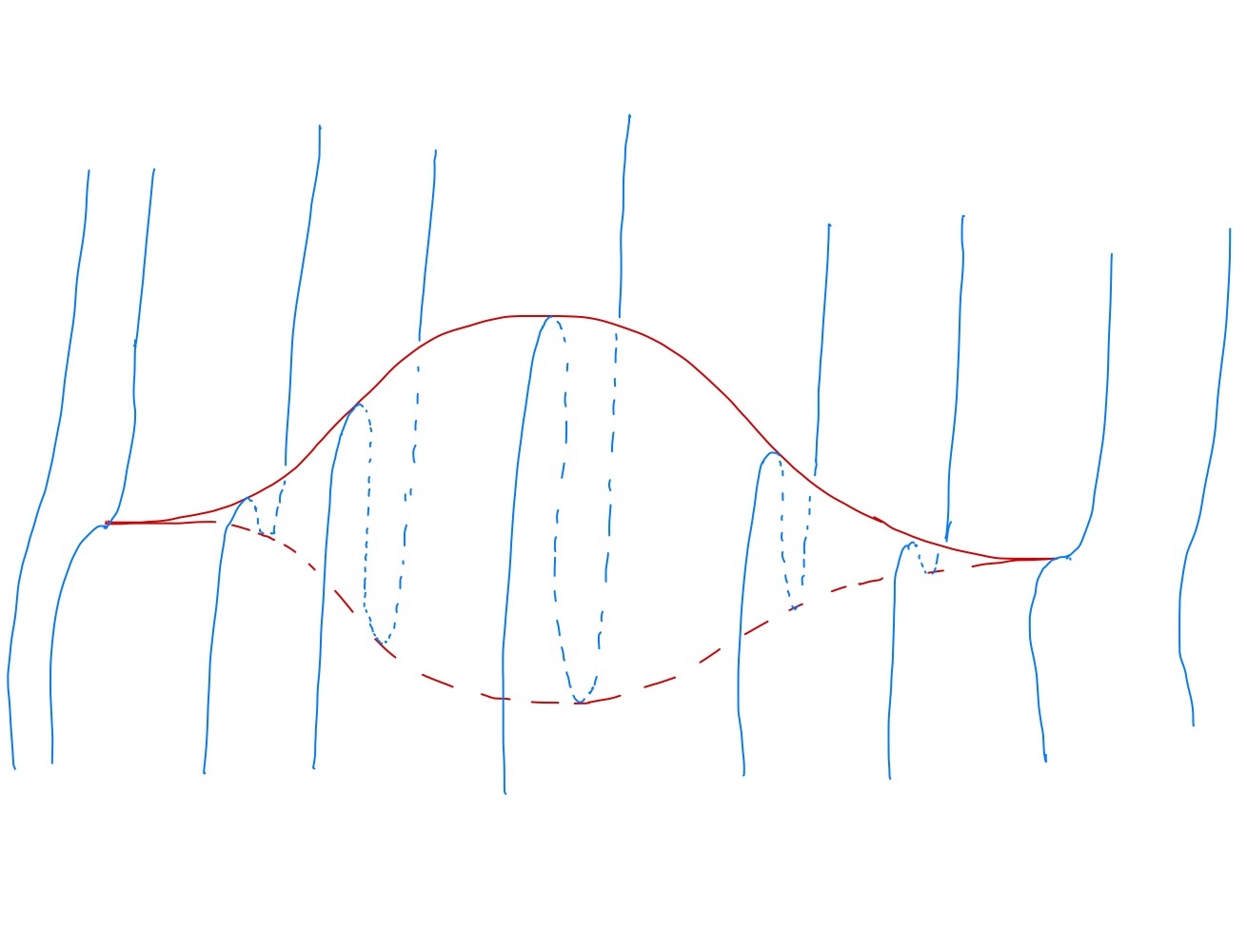} 
\caption{A local deformation of a surface produces a ``pocket'', adding an unknotted component to the collection of fold curves.}
\label{pocket}
\end{figure}

\subsection{Basic spheres}
The second construction gives an explicit embedded 2-sphere in $\R^3$ that has a single fold curve with two cuspidal points, as in Figure~\ref{S2cusps}.
We call this a {\em basic sphere}.
The fold curve contains two smooth arcs, one with  $\mathbf{h}$ pointing to the exterior and one with  $\mathbf{h}$  pointing to the interior of the sphere. 
This allows for the 1-handle in a fold-curve band sum to be attached to either side of the 2-sphere. 

\begin{figure}[htbp] 
\centering
\includegraphics[width=2in]{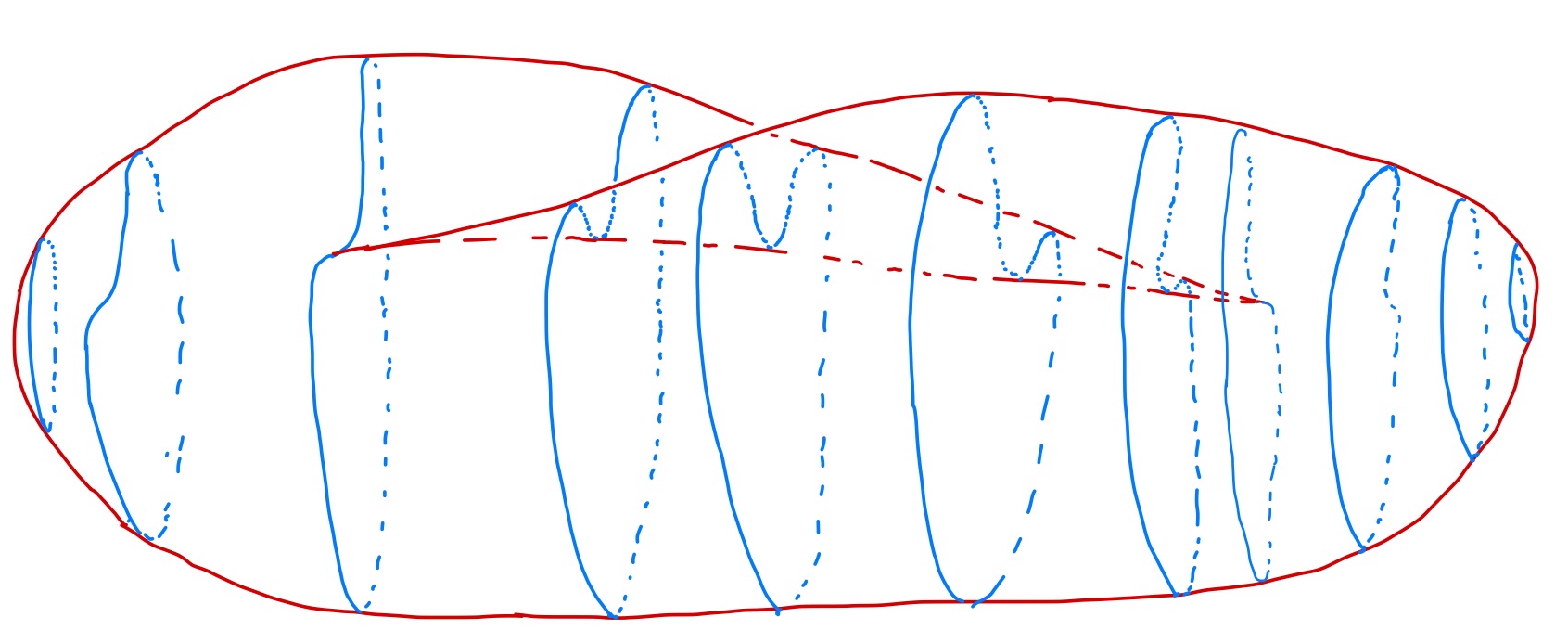} 
\caption{A basic sphere, shown with planar cross-sections, has a single fold curve with two cuspidal points.}
\label{S2cusps}
\end{figure}

To prove Theorem~\ref{unlink} we induct on the surface genus. 
If $F$ has genus zero, Alexander's Theorem implies that there is an isotopy that takes $F$ to a round sphere \cite{Alexander}. 
The fold curves of a round sphere form a single unknot.

Assume now that $F \subset \R^3$ has genus $g > 0$. 
From Lemma~\ref{tubes} we know that $F$ can be obtained from $k$ disjoint 2-spheres
and successively adding $(g+k-1)$ 1-handles, where $1 \le k \le g$.
There is an isotopy that  replaces each of these $k$ spheres by a basic sphere, each basic sphere containing a single unknotted fold curve.
The collection of fold curves is then a $k$-component unlink.

We now construct a surface isotopic to $F$ by adding each 1-handle in the handle decomposition of $F$ using a fold-curve band sum along its core $\alpha$. The 1-handle is attached to a
pre-existing fold curve $\beta_1$ at one end $\alpha(0)$ and to the fold curve $\beta_2$ of a newly created  pocket at the endpoint $\alpha(1)$ of $\alpha$.  
The number of components in the link remains the same, with the component $\beta_1$ connected by a band to the newly created unknot $\beta_2$.  The 1-handle is disjoint from the disk spanning the newly created $\beta_2$ on its basic sphere, and therefore the resulting curve is a connected sum of $\beta_1$  with a trivial knot $\beta_2$ and thus isotopic to $\beta_1$.  The collection of fold curves remains a $k$-component unlink.  Repeating this process for each of the $(g+k-1)$ 1-handles ends in a surface whose fold curves remain a $k$-component unlink.  Since $k \le g$, the collection of fold curves  has at most $g$ components as claimed. Together with the genus-zero case, this completes the proof of the theorem.
 \end{proof} 
 
 \begin{corollary}
 Any torus can be isotoped so that its full collection of fold curves consists of a single unknotted curve. 
 \end{corollary}
 \begin{proof} 
 A torus in $\R^3$ compresses to a 2-sphere after one compression. Therefore the torus is isotopic to a surface made by adding a 1-handle to a single 2-sphere.  This results in a single fold curve.
 \end{proof} 

Figure~\ref{torusOneCurve} shows this process applied to the peripheral torus of a trefoil knot.   
After an isotopy to the torus in the figure, the profile curve  consists of a single unknotted curve.
\begin{figure}[htbp] 
\centering
\includegraphics[width=1.75in]{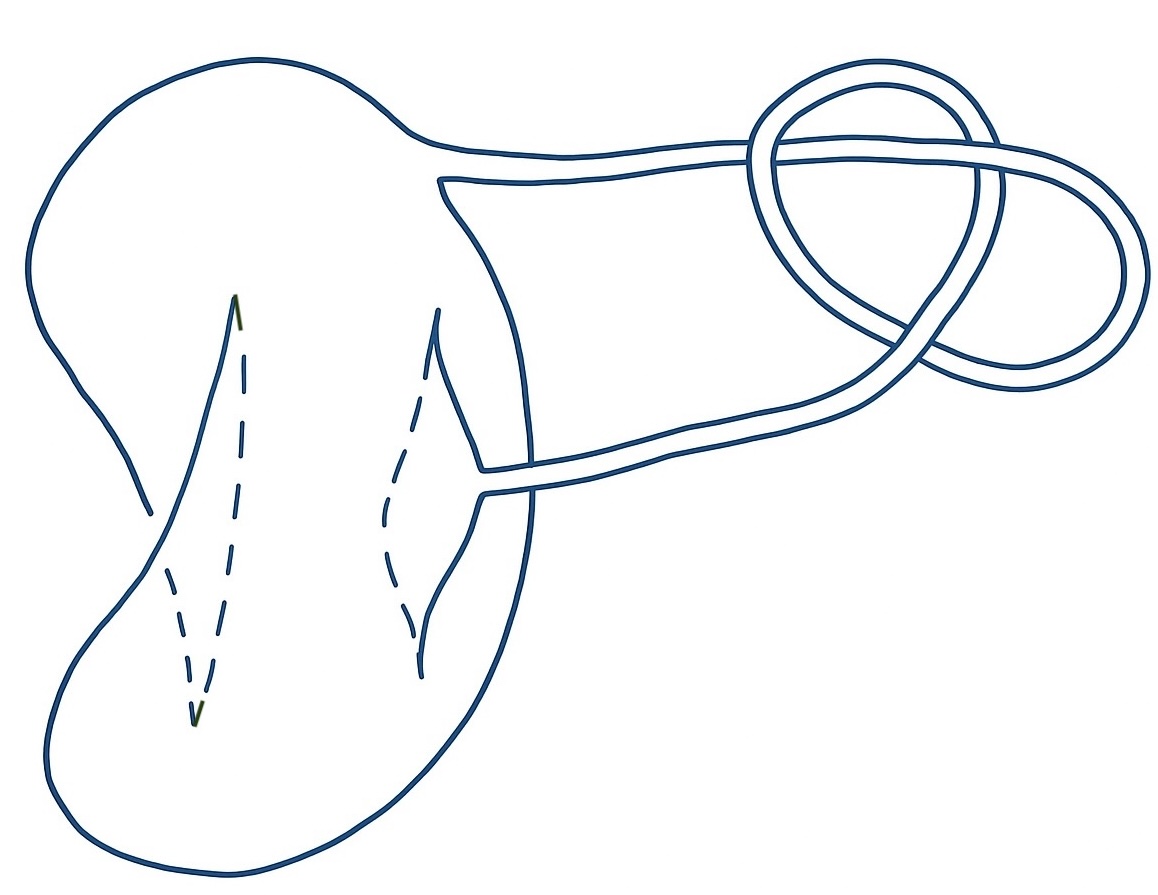} 
\caption{A knotted torus whose fold curves form an unknot. }
\label{torusOneCurve}
\end{figure}

\begin{corollary} \label{1pshere}
Suppose that $F$ is a genus $g$ surface obtained by starting with a 2-sphere and successively adding tubes.
Then $F$ can be isotoped so that its full collection of fold curves consists of a single unknotted curve. 
 \end{corollary}
 The  surfaces described in Corollary~\ref{1pshere} include a  large class of knotted surfaces.  This class contains the boundaries of all knotted handlebodies,
 and also more general surfaces constructed by iteratively adding 1-handles to a sphere, allowing  successive handles to pass through previous handles.


 \section{Geometric Obstructions} \label{CurvetoFold}
 
We now look at the question of when a curve can be made into a fold curve. 
Suppose that $\gamma$ is an oriented regular curve that lies on a surface $F$ embedded in $\R^3$.
Let $H_t \colon  \R^3 \to \R^3$ be an isotopy with $H_0 $ the identity map and $F_t = H_t(F)$.
Does there exist  such an isotopy  leaving $\gamma$ fixed throughout, such that $\gamma$ is a fold curve of $F_1$?
When such an isotopy exists we say that $\gamma$ is {\em realizable as a fold curve}.
We define an invariant $\delta(\gamma, F)$, depending on the curve $\gamma$ and the surface $F$,
that gives a complete obstruction to finding such an isotopy.
In  Theorem~\ref{delta} we show that the curve $\gamma$ is realizable as a fold curve if and only if $\delta(\gamma, F) = 0$.

For the standard $(3,2)$ torus knot lying on an unknotted torus in $\R^3$, as on the left in Figure~\ref{profiletrefoil},
a computation shows that the $\delta$ invariant is equal to three.  Thus this curve cannot be 
transformed into a fold curve by an isotopy of this torus.  
In contrast, the trefoil curve shown on the right in Figure~\ref{profiletrefoil}  has $\delta$ invariant equal to zero, and after an isotopy of the torus it becomes
a fold curve. 

 \begin{figure}[htbp] 
\centering
\includegraphics[width=1.5in]{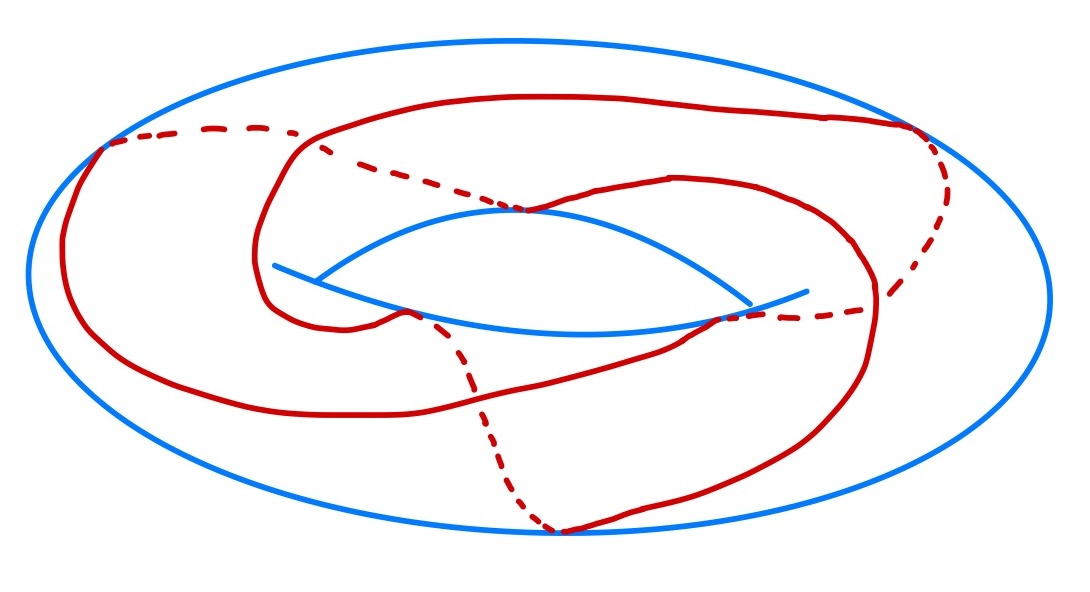} \includegraphics[width=1.5in]{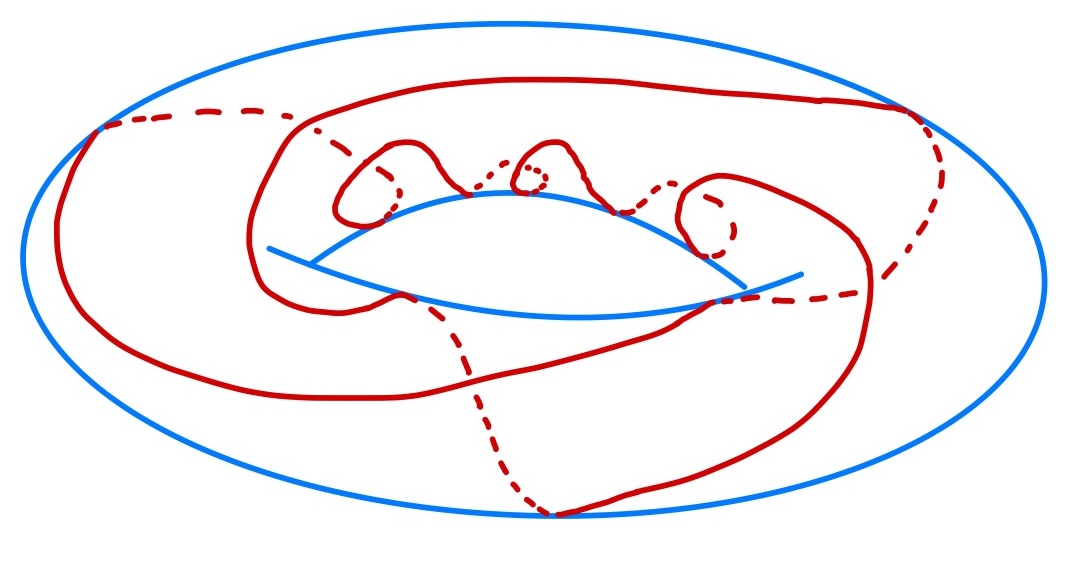} 
\caption{The standard (3,2) torus knot at left is not realizable as a fold curve that lies on an unknotted torus. The isotopic knot at right 
 becomes a fold curve after an isotopy of the torus fixing the curve.}
\label{profiletrefoil}
\end{figure}

We now give a definition of the writhe of a curve in $\R^3$ that applies to curves with cuspidal points.  
Given a simple curve $\gamma$ in $\R^3$ that has no vertical tangents,
each point of $\gamma$ has a unique horizontal unit vector $\mathbf{n_h}$  that is normal to $\mathbf{\dot \gamma}$ 
and for which $[\mathbf{\dot \gamma} , \mathbf{n_h}, \partial/\partial z]$ is a positively oriented basis for $\R^3$.
The {\em writhe} $ w(\gamma)$ of $\gamma$  relative to projection to the $xy$-plane is the linking number of $\gamma$ with a parallel curve
$ \gamma_+ = \gamma + \epsilon \mathbf{n_h}$ , for small $\epsilon$. 
This quantity is independent of the orientation of $\gamma$, and of $\epsilon$ for $\epsilon$ sufficiently small.  Moreover it is unchanged if we
replace  $\mathbf{n_h}$ with  $-\mathbf{n_h}$, since the linking number of a pair of curves is symmetric, and the pair of curves $( \gamma,  \gamma_+)$ is isotopic to the
pair $( \gamma_-, \gamma)$.
The writhe can also be computed by summing the signed crossings of the projection of $\gamma$ to the $xy$-plane.

Now consider  curves that may have isolated vertical tangents. We assume that at points where the curve is vertical it has nonzero  curvature.  
Note that a fold curve of a generic surface has this property.
For such curves the vector field $\mathbf{n_h}$  is not defined at a point of vertical tangency. 
We instead consider the line field $ \nu$ that is horizontal, normal  to $\mathbf{\dot \gamma}$ and $\partial/\partial z$,   
and is normal to the osculating plane at a vertical tangency of $ \gamma$. 
The assumption of nonzero curvature guarantees that this osculating plane exists and extends continuously to a neighborhood of a point on a curve.
Moreover the normal line to this plane is consistent with the line field defined previously at non-vertical tangencies.
Form a curve $\bar \gamma$  by taking both the $+\epsilon$-pushoff and the
 $-\epsilon$-pushoff  of $\gamma$ along
$\nu$. Then $\bar \gamma$  has either one or two components and we can define
the writhe of $ \gamma$  to be half the sum of the linking numbers of  $\gamma$ with the  components of $\bar \gamma$.

The writhe of a curve might  be a half-integer if the line-field $\nu$ is non-trivial, and its value agrees with the standard definition if a curve has no vertical tangencies.
Note that this definition extends the standard definition to the case of fold curves with cuspidal points. 

Now let $F \subset \R^3$ be a smooth embedded oriented surface with normal unit vector field $\mathbf{n_F}$
and $\gamma \subset F$ a curve.
The  surface-normal framing  along $\gamma$ is given by the restriction of $\mathbf{n_F}$ to $\gamma$
and  can be used to produce a push-off curve $\gamma_{F}^+  = \gamma + \epsilon \mathbf{n_F}$, with  $\epsilon>0$ a small constant. 
The {\em surface linking number} $\lambda (\gamma, F)$ of $\gamma$ in $F$ is the linking number of $\gamma$ and $\gamma_{F}^+$,
and measures the twisting of an annular neighborhood of $\gamma$ on $F$ around $\gamma$.  
It can be computed by counting the algebraic intersection number of
$\gamma_{F}^+$  and an oriented Seifert surface for $\gamma$.
Note that the value of  $\lambda (\gamma, F)$ is independent of 
the orientations of $F$ or $\gamma$, and of 
$\epsilon$ for $\epsilon$ sufficiently small. Moreover  an isotopy of $F$ leaves $\lambda (\gamma, F)$ unchanged.
We note that the value of $\lambda (\gamma, F)$ is the same as the value of the Seifert form of $F$ evaluated on $\gamma$.

The next lemma is a standard result about essential annuli in solid tori.  A proof is given for completeness.
\begin{lemma} \label{annuli}
Let $A, B \subset  S^1 \times D^2$ be a pair of annuli, properly embedded in a solid torus $X \subset \R^3$, whose intersection contains a curve  $\gamma$  isotopic to 
the core curve  $S^1 \times (0,0) \subset  S^1 \times D^2$.  Suppose that $\lambda (\gamma, A)$  = $\lambda (\gamma, B)$. Then there is a proper isotopy of the solid torus that carries $A$ to $B$ while fixing $\gamma$.
\end{lemma}
\begin{proof} 

The two boundary curves  $a_1$ and $a_2$ of $A$ are parallel disjoint curves on the torus $\partial X$.  
Similarly the two boundary curves  $b_1$ and $b_2$ of $B$ are parallel disjoint curves on $\partial X$, though $b_1$ and $b_2$  may intersect $a_1$ and $a_2$. 
We choose a basis for $\partial X$ in   which the longitude has linking number zero with the core $\gamma$. 
Since a boundary curve of $A$ is a pushoff of $\gamma$ along $A$, its linking number with  $\gamma$ is given by  $\lambda (\gamma, A)$.  So each
of  $a_1, a_2$  represents one longitude and $\pm m$ meridians in the standard homology basis for $\partial X$, with $m = \lambda (\gamma, A)$,
 Each of $b_1, b_2$ also represents this homology class up to sign, since  $\lambda (\gamma, A) =  \lambda (\gamma, B)$. 
We orient these curves so that each represents  one longitude and $ m$ meridians. 
Two simple curves representing the same homology class on a torus are isotopic, so each of the four curves lies in the same  isotopy class of  curves on  $\partial X$. 

An isotopy supported in a neighborhood of  $\partial X$  carries  each boundary curve $a_1, a_2$ of $A$ to a curve on $\partial X$ that intersects each meridian disk  
$\{\mbox{point}\} \times D^2$ in a single  point.  
 We can then move $B$ by a second isotopy supported in a neighborhood of $\partial X$
so that its boundary curves  $b_1, b_2$ are parallel and disjoint from those of $\partial A$. 
Moreover we can place  $b_1, b_2$  so that a  meridian curve
alternately meets the boundary curves of $A$ and  $B$.
Then each of  $a_1, a_2, b_1, b_2$  intersects each meridian disk  
$\{\mbox{point}\} \times D^2$ in a single  point  
and points of $\partial A$ and  $\partial B$ alternate along this meridian curve.

The annuli $A$ and $B$ can be made transverse to the meridian disk $D= \{0\} \times D^2$  by small isotopies of $A$ and of $B$ that fix $\partial X$ and $\gamma$.
Since each boundary curve of $A \cup B $ intersects the meridian $\partial D$ in a single point, the intersection $A \cap D$ consists of the arc $a$ along with a collection of disjoint embedded loops. 
Similarly $B \cap D$ contains the arc $b$ along with a collection of disjoint embedded loops. While each of  $A \cap D$  and $B \cap D$ form a collection of disjoint simple loops and one arc, 
these collections of curves intersect each other at $a \cap b$ and possibly also at other points where $A \cap B $ meets $D$.

Since the boundary points of $a$ and $b$ alternate along $\partial D$, there is an isotopy of $a$, rel$((a \cap \gamma) \cup \partial a)$, to a curve in $D$ that meets the arc $b$ only at the point $a \cap \gamma = b\cap \gamma$.
This isotopy extends to an isotopy of $A$ supported in a neighborhood of $D$.   After this isotopy the arcs $a$ and $b$ meet in exactly one point, where they meet $\gamma \cap D$.
The arc  $a$ now divides $D$ into two complementary regions.  An innermost loop of $A \cap D$ on $D$ lies in one of these two regions and bounds a disk on $D$ with interior in the complement of $A$.
Since each curve in $\partial A$ is homotopic to the core of $X$, the annulus $A$ has  fundamental group that injects into $\pi_1(X)$ and is incompressible. So the innermost loop also bounds a disk on $A$, and the union of these two disks forms a 2-sphere. Since  $S^1 \times D^2$ is irreducible, this 2-sphere bounds a ball and therefore the two disks lie on the boundary of a 3-ball and are isotopic. 
Applying this isotopy to move $A$ across the innermost disk removes the  innermost loop from $A \cap D$.  Repeating this process isotops $A $ until all loops in $A \cap D$ are eliminated.  We then do the same for $B$.
The result is a series of isotopies that lead to the
intersection $(A \cup B) \cap D$ consisting entirely of the two arcs $a,b$, each embedded and meeting only at $\gamma \cap D$.

Cutting  $X$ open along $D$ results in a ball $B$ that contains embedded disks $D_A$ and $D_B$ obtained from cutting  the annuli $ A, B $ along $a,b$ respectively. 
There are two copies $D_1, D_2$  of the disk $D$ in $B$ and a homeomorphism $g:D _1 \to D_2$ that glues them to recover $X$, $A$ and $B$.
The disks $D_A$ and $D_B$ intersect along the proper arc $\gamma_B$ resulting from cutting $\gamma$ at the point $D \cap \gamma$ and possibly also at
additional points. Note that $D_A$ and $D_B$ may intersect non-transversely along  $\gamma_B$.  On the boundary of $B$ the intersection $\partial B \cap (D_A \cup D_B) $ 
consists of two loops $\partial D_A$ and $\partial D_B$ intersecting transversely at two points.  The  2-sphere $\partial B$  is separated by  $\partial D_A \cup \partial D_B$  into four bigons, and these can be used to 
construct an isotopy of $D_B$ that carries $\partial D_B$ to $\partial D_A$ and fixes  $D_B \cap D_A$. 
The disk $D_A$ and $D_B$ then have the same boundary on $B$ and intersect on a proper arc that cuts each of them into two subdisks,
${D_A}^-, {D_A}^+, {D_B}^-, {D_B}^+$.
Irreducibility of the 3-ball implies that ${D_A}^-, {D_B}^- $. and $ {D_A}^+,  {D_B}^+$. are isotopic (rel boundary).
These isotopies of $B$ can be chosen to be symmetric on the two copies $D_1, D_2$  of  $D$,
so that they respect the gluing map $g:D _1 \to D_2$. The isotopy then extends to an isotopy of $X$, and carries the annulus $A$ to the annulus $B$, as claimed.
\end{proof} 

The following lemma captures a useful property of $\lambda (\gamma, F)$.

\begin{lemma} \label{twists}
Let $F, G \subset \R^3$ be a pair of smooth embedded oriented surfaces and $\gamma \subset F \cap G $ a curve lying  on each of $F$ and $G$.
The following statements are equivalent:
\begin{enumerate}
\item There is an isotopy $H_t \colon  \R^3 \to \R^3,   ~0 \le t \le 1$, from the identity map $H_0 $ 
to a diffeomorphism $H_1$ such that $H_t$
leaves $\gamma$ fixed and such that  $H_1(F)$ coincides with $G$ in an open annular neighborhood of $\gamma$.
\item  $\lambda (\gamma, F) =  \lambda (\gamma, G)$. 
\end{enumerate}
\end{lemma}
\begin{proof} 
The value of  $\lambda (\gamma, F)$ is determined by an open annular neighborhood of $\gamma$ in $F$ and is preserved by ambient isotopy.
So if there is an isotopy of $F$ that makes the two surfaces coincide in a neighborhood of $\gamma$ then  $\lambda (\gamma, F) =  \lambda (\gamma, G)$. 

Conversely, suppose that $\lambda (\gamma, F) =  \lambda (\gamma, G)$. 
Choose  $\epsilon>0$ sufficiently small so that the $\epsilon$-tubular neighborhood $N_\epsilon  (\gamma) $ of $\gamma$ in $\R^3$ is a solid torus containing proper annuli 
$A_F = F \cap N_\epsilon  (\gamma)$ and $A_G = G \cap N_\epsilon  (\gamma)$. 
Each of $A_F,  A_G $ contains a core curve of $N_\epsilon  (\gamma) $,  and the value of $\lambda$ is determined by a neighborhood of $\gamma$ in a surface, so
 $\lambda (\gamma, A_F) =  \lambda (\gamma, F) = \lambda (\gamma, G) =  \lambda (\gamma, A_G)$
and the hypotheses of Lemma~\ref{annuli} are satisfied.
It follows that the annuli  $A_F, A_G$  are  properly isotopic in $N_\epsilon$ by an isotopy that fixes their common intersection curve $\gamma$.
Using  the isotopy extension theorem we can  extend this isotopy from $N_\epsilon (\gamma)$  to $\R^3$. We then obtain an
isotopy $H_t \colon  \R^3 \to \R^3$  from the identity map to a map  $H_1$ such that $H_1(F)$ coincides with $G$ in $N_\epsilon (\gamma)$. 
This proves the lemma.
\end{proof} 

\noindent
{\bf Definition.} The {\em $\delta$ invariant}  of an embedded curve $\gamma$ lying on a surface $F \subset \R^3$  is
$$
\delta(\gamma, F) = \lambda (\gamma, F) - w(\gamma).
$$

The definition gives a number that lies in $(1/2)\ZZ$.
We now show that this invariant determines whether a curve can be made into a fold curve for a stable surface.

A {\em potential fold curve} is a curve in $\R^3$  that is a fold curve for some annulus $A$.
Any curve with no vertical tangencies is a potential fold curve. This class of curves also contains certain curves with isolated vertical tangencies,
which turn into cuspidal points.
Curves that are ruled out by this definition include, for example, those with non-isolated vertical tangencies.

The next theorem  shows that there is an isotopy of $F$ that  turns $\gamma$  into a fold curve exactly when $\delta(\gamma, F) = 0$. 

\begin{theorem} \label{delta}
A potential fold curve  $\gamma$  that lies on an embedded surface $F \subset \R^3$ is realizable as a 
fold curve if and only if $\delta(\gamma, F) = 0$. 
\end{theorem}
 
\begin{proof}
Assume first that  $\gamma$ is realizable as a fold curve. 
At each point of $\gamma$ the  normal vector $\mathbf{n_F}$
is a unit-length horizontal vector.  It spans a horizontal line field $ \nu_F$.
Recall that the surface linking number  $\lambda (\gamma, F)$ is independent of the choice of normal for $F$, and is computed by taking the linking number of $\gamma$ with 
either of the two pushoffs $\gamma_{F}^+  = \gamma + \epsilon \mathbf{n_F}$ or $\gamma_{F}^-  = \gamma -\epsilon \mathbf{n_F}$.
So  $\lambda (\gamma, F)$ is equal to half of the linking number of $\gamma$ with the curve $\bar \gamma =  \gamma_{F}^+ \cup \gamma_{F}^- $.
This curve is formed by taking $\epsilon$-sections along the line field $ \nu_F$, exactly as in the definition  of the writhe of $\gamma$, and therefore  $\lambda (\gamma, F)$ 
is equal to the writhe of $\gamma$.  So their difference vanishes and  $\delta(\gamma, F) = 0$.


Conversely, suppose that  $\delta(\gamma, F) = 0$.
Since $\gamma$ is a potential fold curve, there exists an annulus $A$ with $\gamma$ a fold curve of $A$.
The previous argument shows that  $\delta(\gamma, A) = 0$, so that 
 $\delta(\gamma, F) = \delta(\gamma, A) $.
Adding the writhe of $\gamma $ to each side of this equation shows that 
$\lambda (\gamma, F) =  \lambda (\gamma, A)$.  
By Lemma~\ref{twists} there is an isotopy of $F$ fixing $\gamma$ 
that carries an annular neighborhood of $\gamma$ in $F$ to an annular neighborhood of $\gamma$ 
in $A$, and thus turns $\gamma$ into a fold curve for $F$.
  \end{proof}
 
The twisting process  of an annular neighborhood that turns $\gamma$  into a fold curve in Theorem~\ref{delta}  is illustrated in Figure~\ref{twisttorus} for a (1,1) curve on an unknotted torus. 
In this example the writhe is zero and the surface linking number is one, so that $\delta(\gamma, F) = 1$  and $\gamma$ cannot be made into a fold curve. 
The  initial arc transverse to $\gamma$ at  $\gamma(0)$ does not match up with the final arc  transverse to $\gamma$ at  $\gamma(1)$, so
an isotopy of an annular neighborhood is not constructed,
and the torus cannot be isotoped so that this (1,1) curve is a fold curve. However by isotoping the (1,1) 
curve on the torus so that its writhe equals one, we can arrange that $\delta(\gamma, F) = 0 $. Theorem~\ref{delta} then implies
that the second curve in  Figure~\ref{twisttorus} does become a fold curve after an  isotopy of the torus.
 
\begin{figure}[htbp] 
\centering
\includegraphics[width=1.5in]{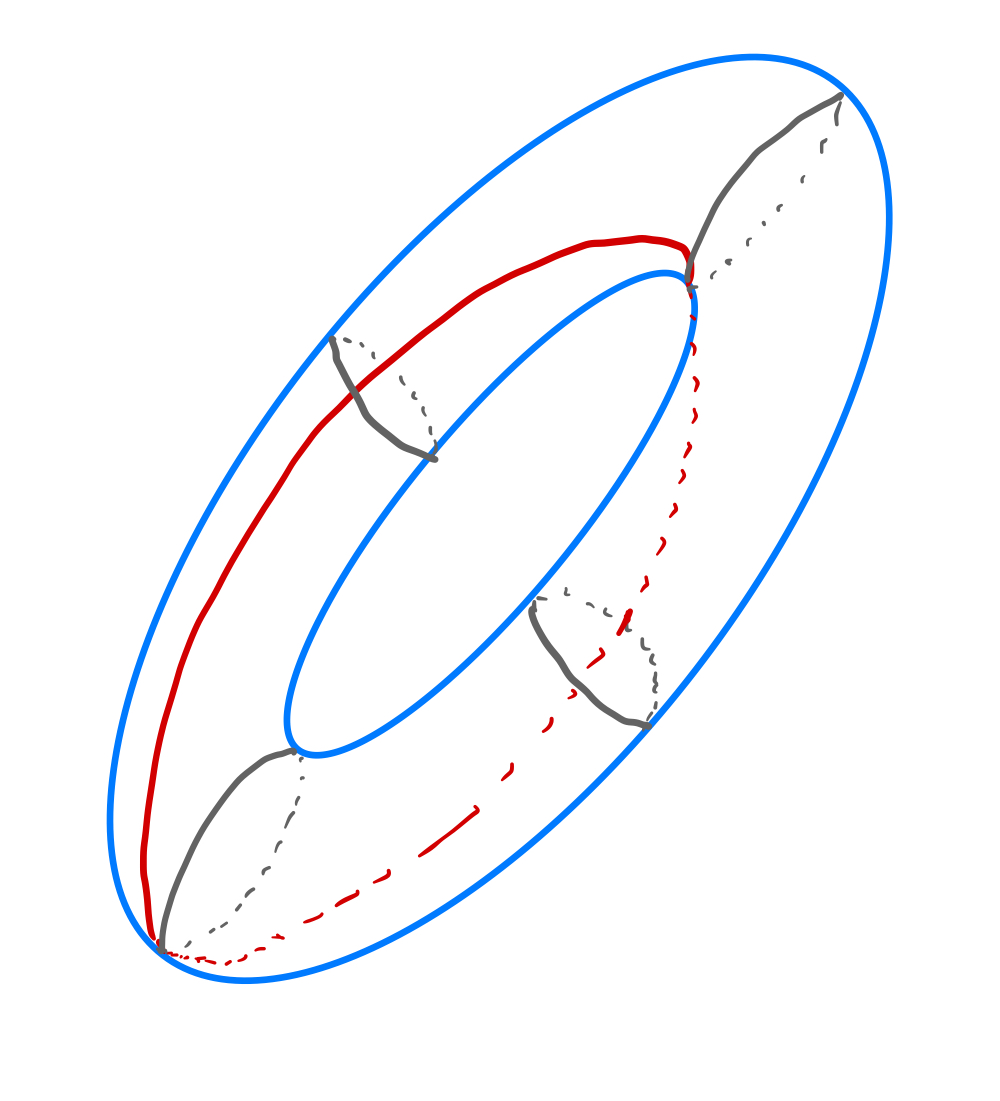} \hspace{1in} \includegraphics[width=1.5in]{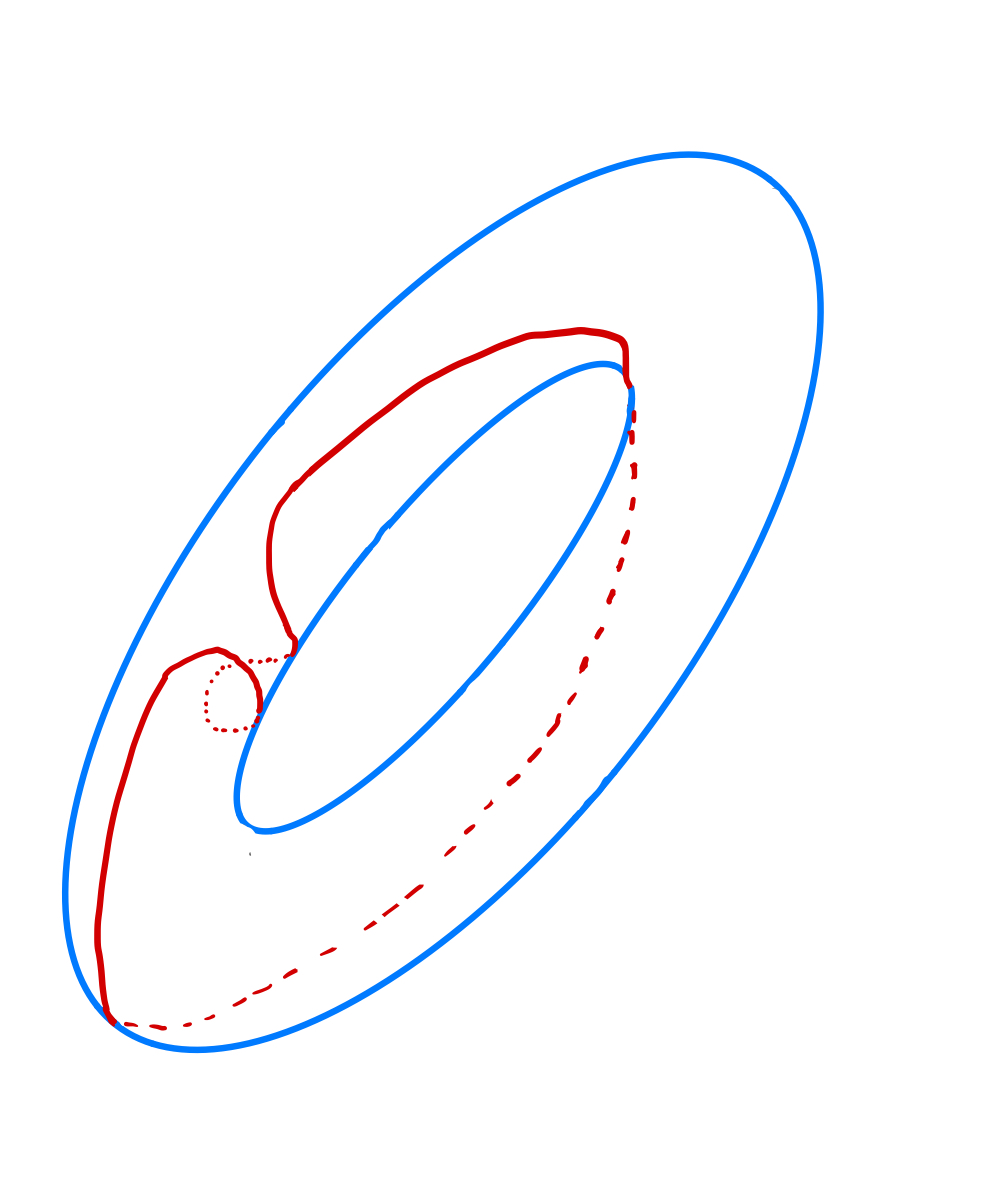}  \\
\includegraphics[width=0.8in]{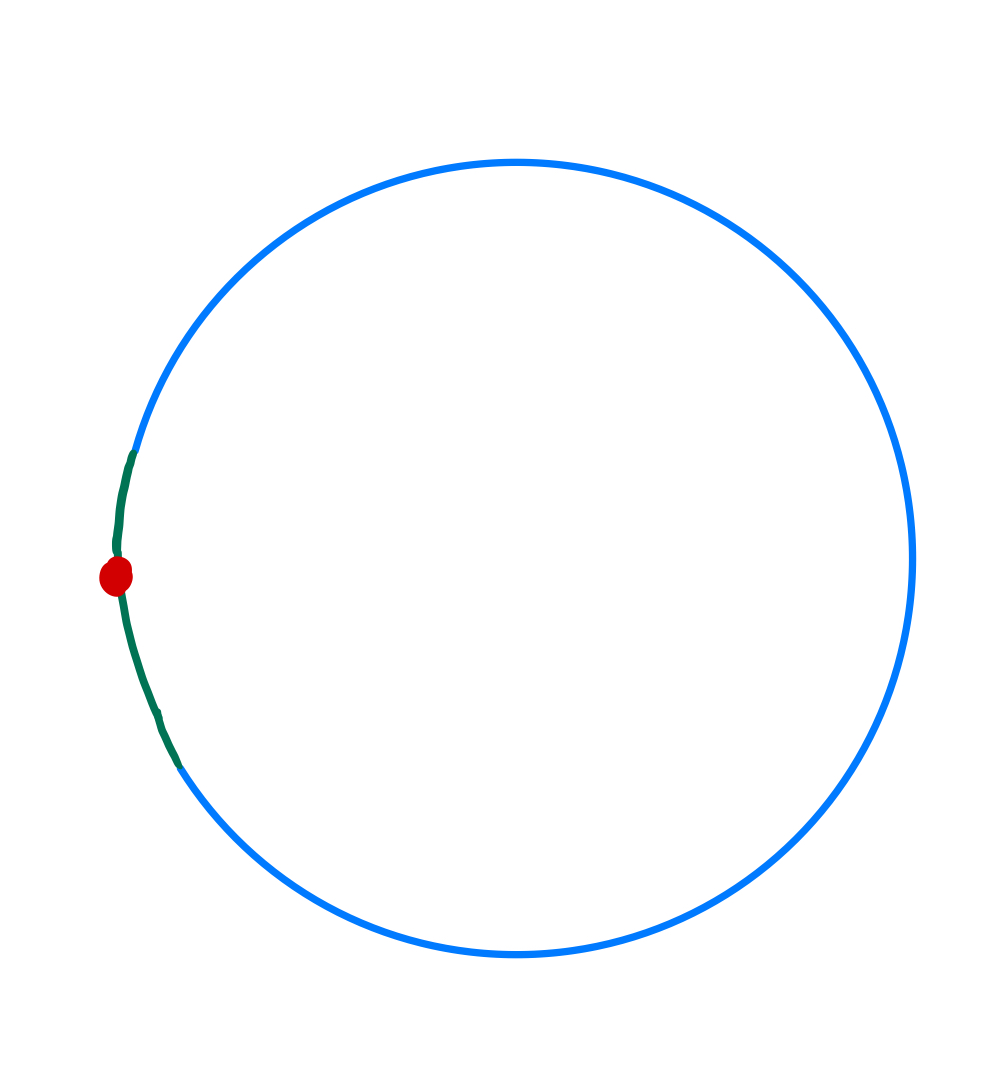} \includegraphics[width=0.7in]{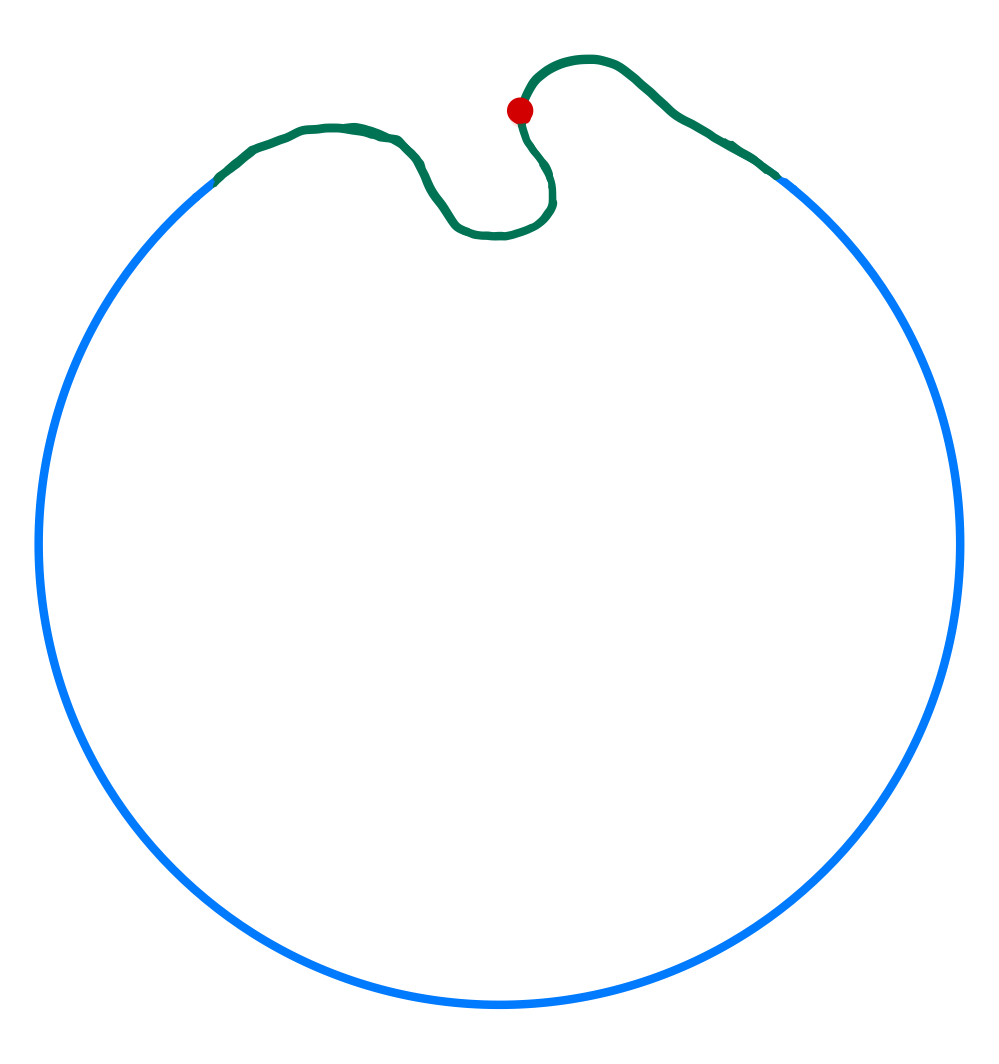} \includegraphics[width=0.85in]{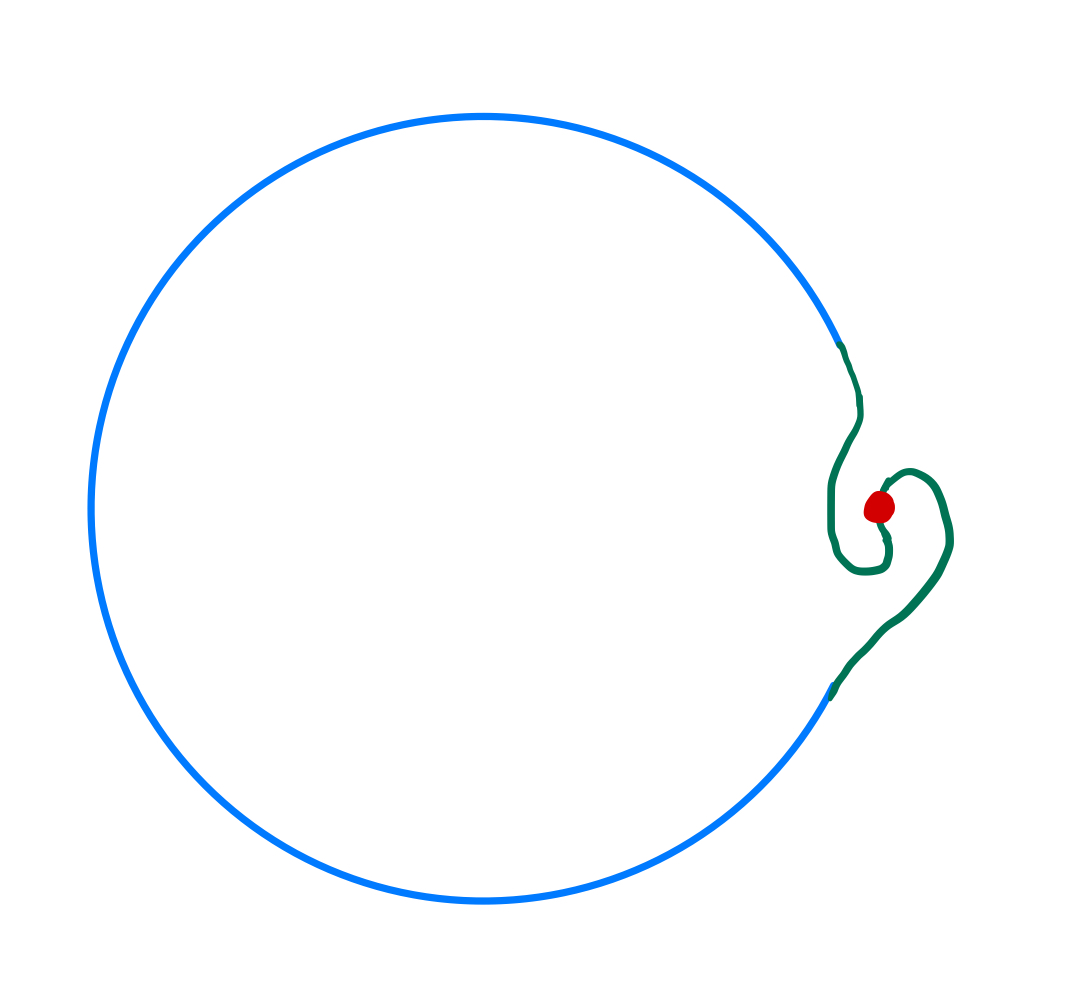} \includegraphics[width=0.7in]{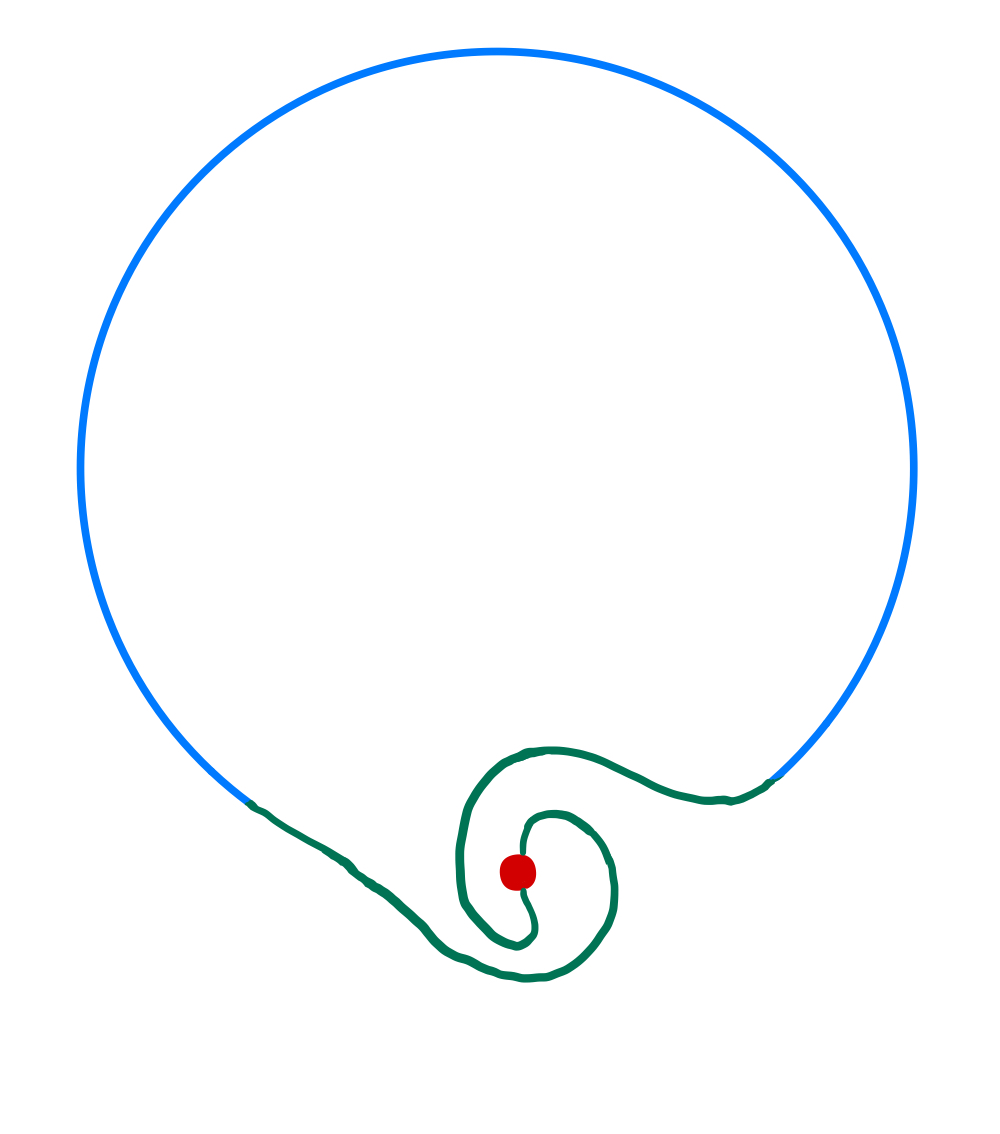} \includegraphics[width=0.9in]{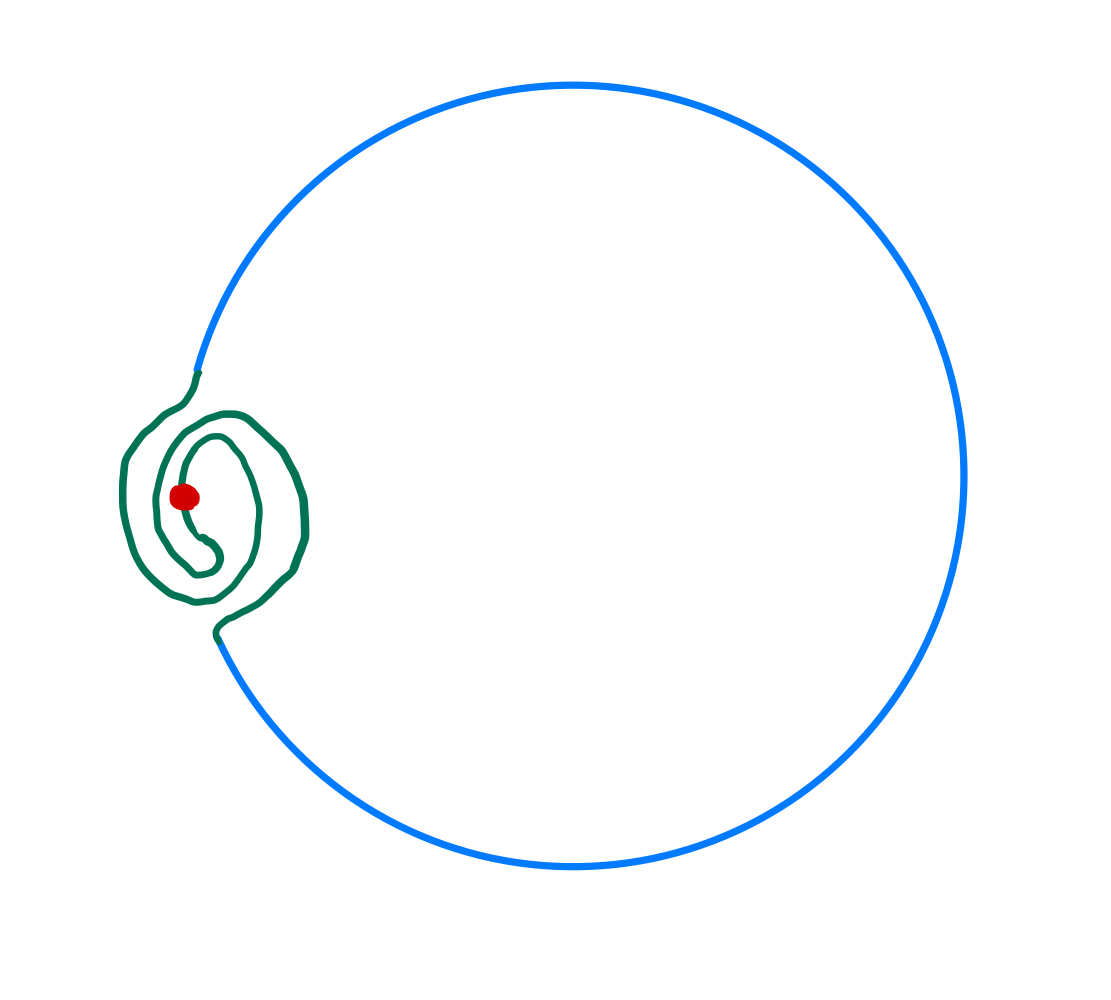} 
\caption{A thin annular neighborhood of a  (1,1) curve $\gamma$ on a torus is twisted so that the tangent plane along $\gamma$  is vertical and $\gamma$  becomes a fold curve. 
The last cross-section does not match the first so the resulting surface fails to have $\gamma$ as a fold curve. 
An isotopy of the curve that changes its writhe, as in the curve at right, resolves this issue.} 
\label{twisttorus}
\end{figure}
 
 In contrast to Theorem~\ref{delta}, if an isotopy of $F$ is not required to fix $\gamma$ then $F$ can always be moved to make  $\gamma$ into a fold curve.
  
\begin{corollary} \label{afterisotopy}
Let $\gamma_0$ be an embedded curve on an embedded surface $F_0 \subset \R^3$. Then 
there is an isotopic curve
$\gamma_1 \subset F_0$   and an isotopy of $F_0$,  fixing $\gamma_1 $, taking  $F_0$  to a surface $F_1$
for which $\gamma_1 $ is a fold curve.
\end{corollary}
\begin{proof}
A small isotopy moves $\gamma_0$ on $F$ to a curve $\gamma'$ with no vertical tangents.
Isotopies on $F$ preserve  $\lambda (\gamma, F)$, so we restrict attention to how an isotopy on $F$ affects the writhe $w(\gamma')$.
We describe how to construct an isotopy  of $\gamma' $ on $F$ which leads to a series of 
type-I Reidemeister moves on the projection of $\gamma'$  that change its writhe and
end with a curve $\gamma_1 \subset F$ for which the writhe equals  $\lambda$.

These moves add twists to an arc $\alpha$ of $\gamma'$ that is adjacent to a fold curve of $F$, 
and can be constructed to either add or subtract one from the writhe. 
They are defined as follows.

First we choose an arc $\alpha \subset \gamma'$ and a parallel arc  $\beta$ on  the fold curve. 
By parallel we mean that there is a 
disk embedded in the surface that is homeomorphic to the unit square,  with the arcs $\alpha$ and $\beta$ the images of the bottom and top edges of the square and the interior of the disk disjoint from $\gamma'$ and the fold curves.
We can find such a disk by finding an arc $\delta[0,1] \to F$ from  $\gamma'$  to a fold curve that intersects  $\gamma'$ only at its initial point.  The arc $\delta$ can be used to  guide an isotopy of  $\gamma'$ in which $\delta$  becomes  short, with its initial point  close to a fold curve.  In particular  when $\delta$  is sufficiently short the planar projection of $\delta$  has interior disjoint from that of $\gamma'$.
Now we thicken $\delta$ to obtain a disk  $\delta \times [-1,1]$ embedded on $F$.  
The initial point of the arc thickens to become $\alpha = \delta(0) \times [-1,1]$ and the endpoint thickens to form $\beta= \delta(1) \times [-1,1]$, showing that $\alpha$ and $\beta$ are parallel.

 The fold curve has a canonical orientation, and we use this to orient first $\beta$ and then the parallel arc  $\alpha$.
The surface $F$ has two sheets near $\beta$, one lying above the other.  
Consider first the case where $\alpha$ lies in the upper sheet.
We describe two moves that change the arc  $\alpha$ by a type-I Reidemeister move,  adding a twist to the projection of $\alpha$.  Each move changes the writhe by one,
with one adding 1 to the writhe and the other subtracting one.  The first move isotops $\alpha$  to an arc that loops back and crosses $\beta$ in the negatively oriented direction, crosses underneath itself to produce a positive crossing, then recrosses  $\beta$ in the positively oriented direction and rejoins  $\gamma'$.  The resulting curve, isotopic to $\gamma'$, has its writhe changed by $+1$. See the left figure in  Figure~\ref{Fig:twist}.  The second move isotops  $\alpha$ to an arc that crosses $\beta$ in the positively oriented direction, loops around and recrosses $\beta$ in the negatively oriented direction, and then crosses over itself to rejoin $\alpha$.  This results in an isotopic curve whose  writhe has changed by $-1$. See the right figure in  Figure~\ref{Fig:twist}. Now consider the case where $\alpha$ lies in the lower sheet. The two moves in this case are vertical reflections of the previous case, and change the writhe with opposite sign.  But again one move increases and one move decreases the writhe.

After an isotopy of  $\gamma'$ in $F$ that adds an appropriate number of  positive or negative twists to $\gamma'$ along $\alpha$ by using these moves, the writhe of the resulting curve $\gamma_1$ becomes equal to the surface linking number, and Theorem~\ref{delta}  applies.
Then there is an isotopy of the surface in $\R^3$ that fixes $\gamma_1$ and produces a surface $F_1$ for which $\gamma_1$ is a fold curve.
\end{proof}
 
\begin{figure}[htbp] 
\centering
\includegraphics[width=1.5in]{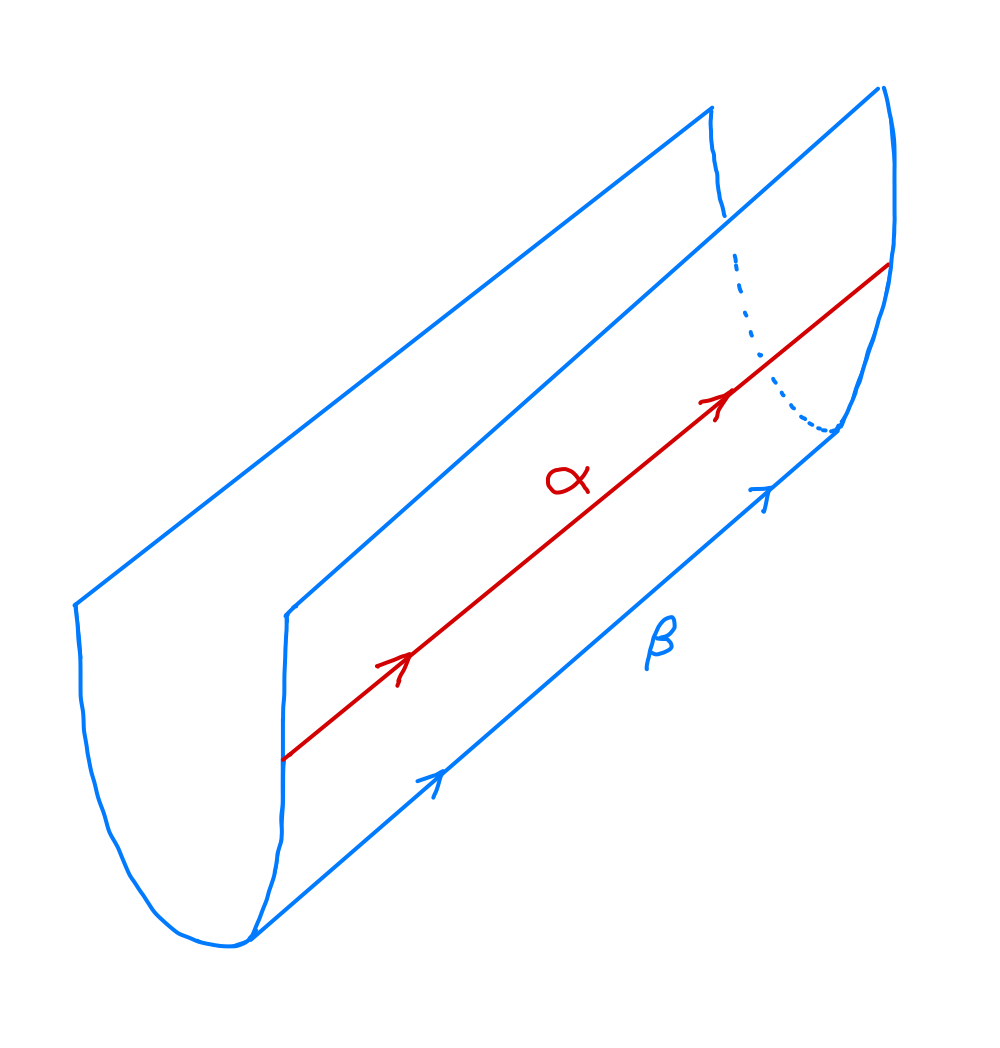}   \ \ \  \includegraphics[width=1.5in]{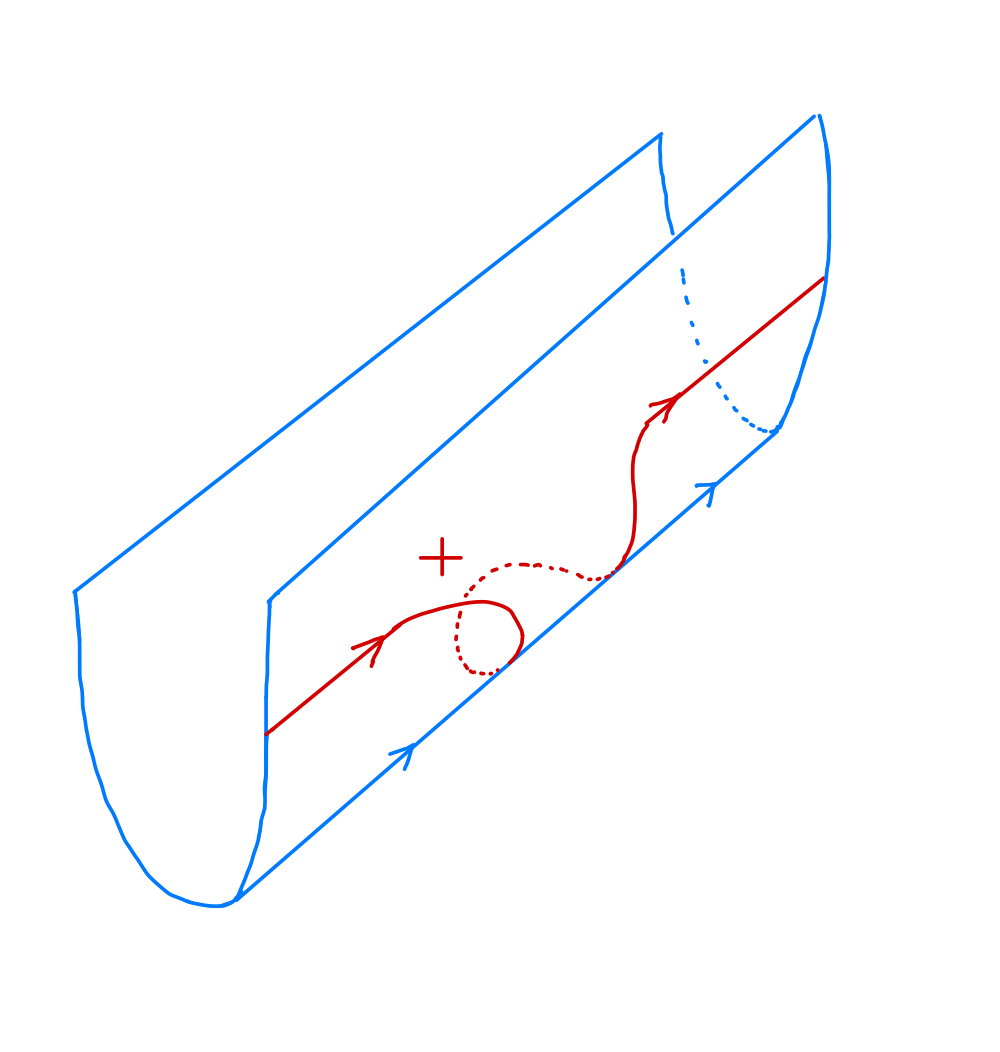}   \ \ \   \includegraphics[width=1.5in]{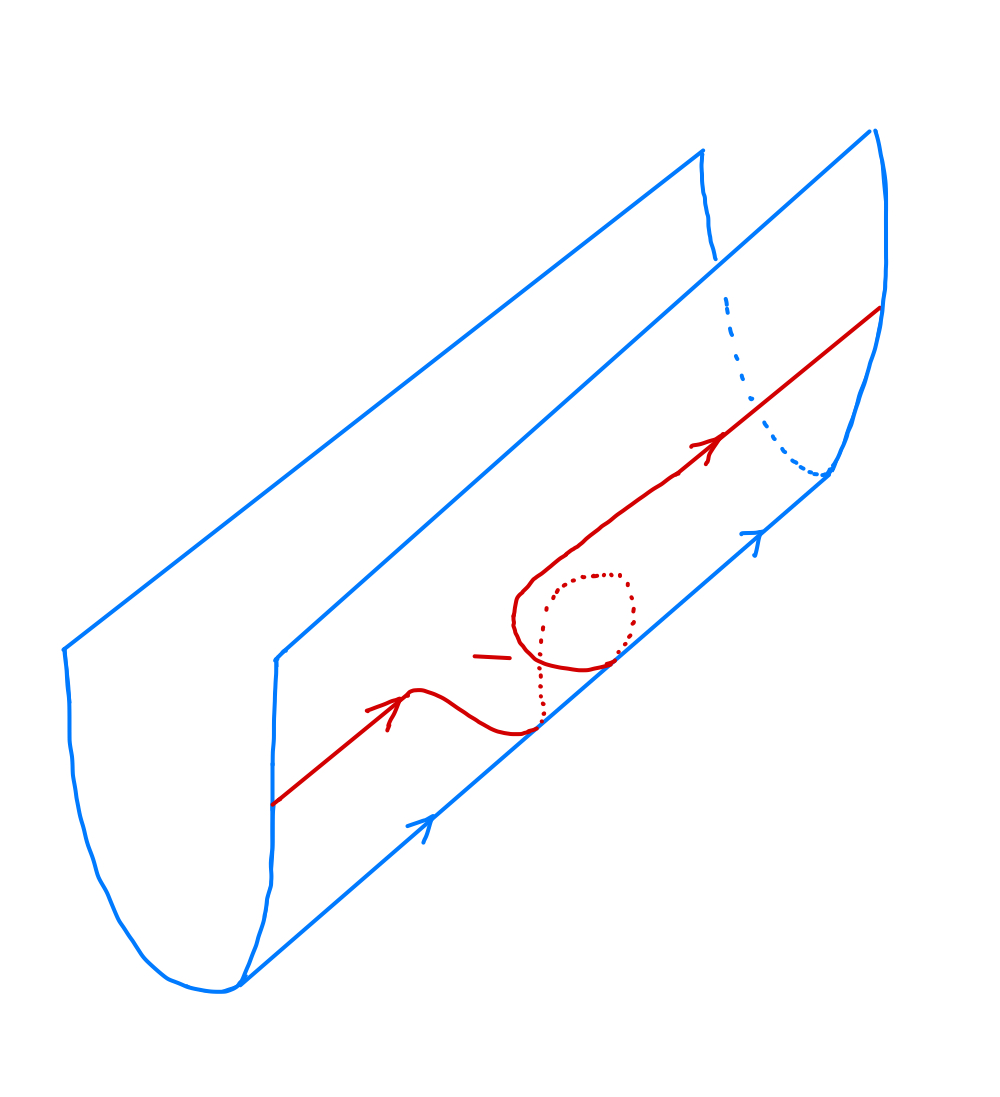}   \ \ \ 
\caption{A curve on a surface can be isotoped across a fold curve so as to change its writhe by $+1$ or $-1$.} 
\label{Fig:twist}
\end{figure}

The  {\em surface genus} of a curve $\gamma$ is the smallest genus of an unknotted surface in $\R^3$ that contains a curve isotopic to $\gamma$. 

\begin{corollary}
A knot $K$ with surface genus $g(K)$ can be realized as a fold curve of an unknotted genus-$g$ surface  $F \subset \R^3$ for any $g \ge g(K)$.
\end{corollary}

\begin{example} \label{multi}
Any torus knot can be realized as a fold curve for an unknotted torus.
\end{example}


\section{Examples} \label{restrictions}

In this section we present  examples of restrictions imposed on an embedded surface $F \subset \R^3$ 
by the condition that $F$ contains a given curve $\gamma$ as a fold curve.
These apply  irrespective of whether there are additional fold curves on the surface.
These restrictions are connected to the isotopy class of $F$ and do not apply if
we let the isotopy class  change, or if we allow $F$ to be immersed rather than
embedded.  In particular, they behave differently from obstructions to lifting generic maps
of a surface into $\R^2$ to maps into  $\R^3$ as studied in \cite{Haefliger, Levine, Millett, Pignoni}.

We illustrate these restrictions with the two curves shown in Figure~\ref{pfcurves}.

\begin{figure}[h]
\centering
\begin{minipage}{.3\textwidth}
\centering
\includegraphics[width=.8\linewidth]{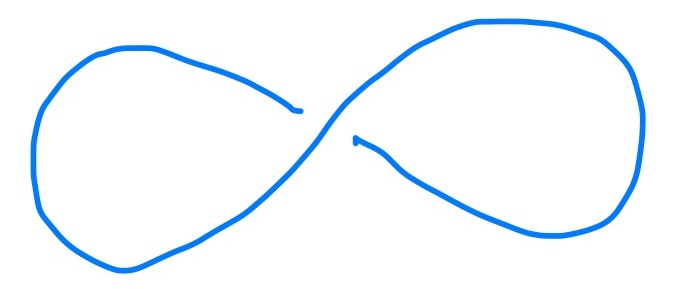}
\end{minipage}
\begin{minipage}{.3\textwidth}
\centering
\includegraphics[width=.8\linewidth]{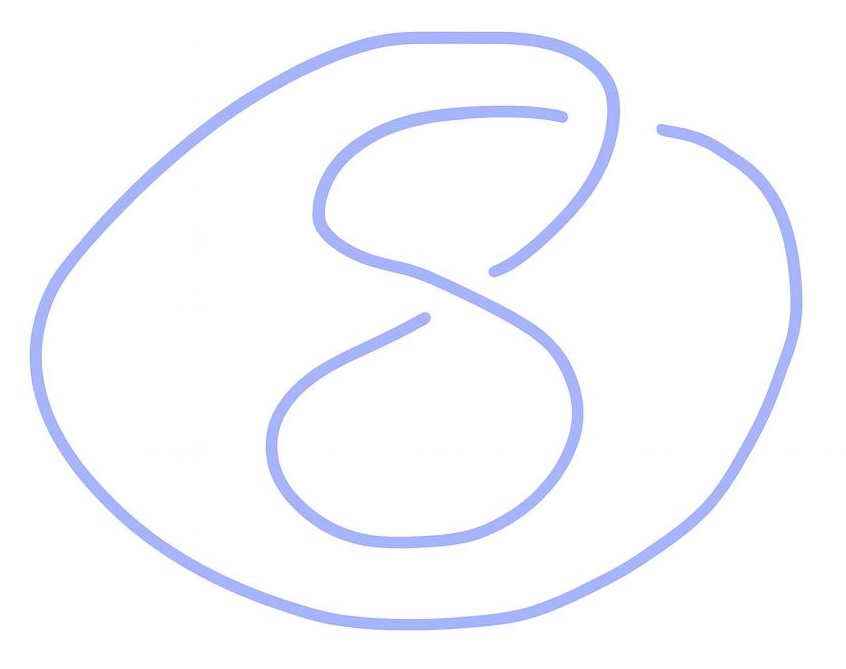}
\end{minipage}
\caption{Two potential fold curves, $\gamma_1$  (left) and $\gamma_2$ (right).}
\label{pfcurves}
\end{figure}

\begin{example}
The curves $\gamma_1$ and $\gamma_2$ in Figure~\ref{pfcurves} cannot be fold curves lying on an embedded sphere.
 \end{example}
 \begin{proof}
 The surface linking number $\lambda$ of a curve lying on a sphere is zero, since the curve
 bounds a disk on the sphere that is disjoint from a normal push-off.
Since $ w(\gamma_1) = 1$ and $ w(\gamma_2) = 2$,
the $\delta$ invariant of each of $\gamma_1$ and $\gamma_2$ is nonzero.
Thus neither can be a fold curve lying on a sphere.
\end{proof}

Note that $\gamma_2$ projects to a curve  in the plane that coincides with 
the apparent contour generated by an embedded 2-sphere in $\R^3$.
Thus the $\delta$ invariant gives information that cannot be deduced from the apparent contour of a surface.

Observations of this type can be applied when studying surfaces of unknown topology. 
A biologist observing $\gamma_1$ or $\gamma_2$ as a fold curve in an image
of the surface of an organism seen through a microscope
can conclude that the surface is not  a sphere.

We now consider two examples where the surface is a torus.
A fold curve that is  an embedded loop in the plane  can be  generated by either a knotted or an unknotted torus. 
In contrast, each of the curves in Figure~\ref{pfcurves} can be a fold curve of an unknotted torus, but not of a knotted torus.

\begin{example}
The curves $\gamma_1$ and  $\gamma_2$   in Figure~\ref{pfcurves} cannot be fold curves lying on a knotted torus.
 \end{example}
\begin{proof}
We show this for $\gamma_1$. 
Suppose that the curve $\gamma_1$ is a fold curve on a torus. 
A curve bounding a disk on a torus has surface linking number zero. 
Since $\gamma_1$ has writhe one and by  Theorem~\ref{delta} 
 has $\delta$ invariant  zero, it
cannot bound a disk on the torus.

Similarly $\gamma_1$ cannot represent a torus meridian, which also has  surface linking number  zero since a meridian bounds an embedded  compressing disk lying in the complement of the torus in $\R^3$. 
Thus $\gamma_1$ is homotopically nontrivial and not a meridian, and thus represents $l$ longitudes and $m$ meridians with $l\ne 0$.   If $l=\pm 1$ then $\gamma_1$ is isotopic to the core of the torus and if the torus is knotted then $\gamma_1$ is also knotted.
If $|l|>1$ then $\gamma_1$ is a non-trivial satellite knot. 
Since $\gamma_1$ itself is unknotted, this implies that $l = \pm 1$ and  the torus itself is unknotted.  The argument for $\gamma_2$ is  similar.
\end{proof}

\begin{example}
If a curve $\gamma$ is isotopic to a  trefoil knot and is a fold curve for a torus, and if the sum of the number of crossings and half the number of cusps in the projection of $\gamma$ is less than six, then the torus is knotted.  \end{example}
\begin{proof}
A crossing contributes $\pm 1$ to the writhe and a cusp contributes $\pm 1/2$.
There are at least three crossings in the projection of $\gamma$, so the condition on the number of crossings and cusps implies that the writhe of $\gamma$ is nonzero (mod $6$).
The surface linking number of a trefoil on the surface of an unknotted torus is zero (mod $6$), as can be seen by computing this for the standard diagram of a $(3,2)$-torus knot.
Since these are not equal, Theorem~\ref{delta} implies that $\gamma$ cannot be a fold curve for an unknotted torus.
Figure~\ref{Fig:six} shows two examples of such curves.
\end{proof}
 
\begin{figure}[htbp] 
\centering
\includegraphics[width=1.1in]{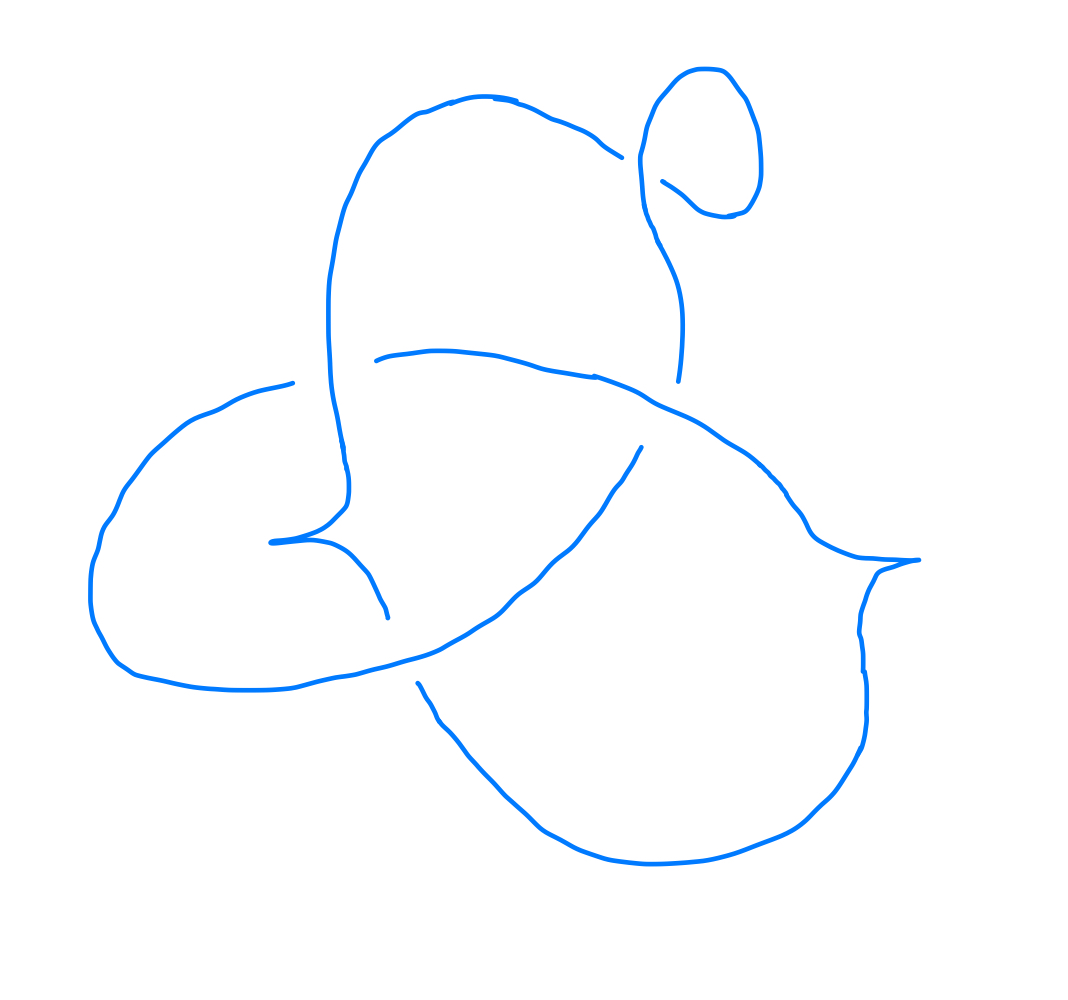} \ \ \ \ \   \ \ \  \includegraphics[width=1.in]{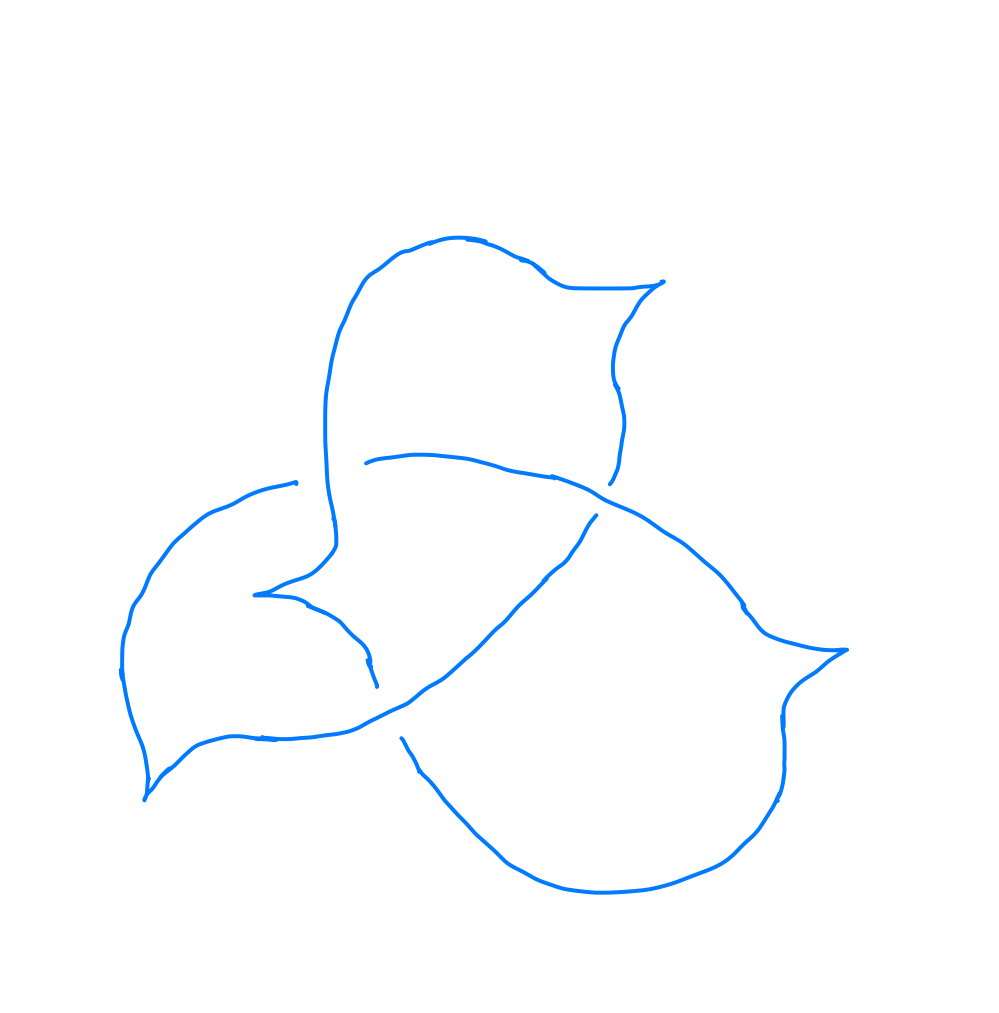} 
\caption{Two curves that cannot be fold curves of an unknotted torus.} 
\label{Fig:six}
\end{figure}
 
\subsection{Questions} \label{open}

We conclude with several questions for further investigation.  

\begin{enumerate}
\item 
Theorem~\ref{unlink} shows that an embedded connected surface $F \subset \R^3$ of genus $g \ge 1 $ can be isotoped so that its full collection of fold curves consists of an unlink with at most $g$ components.  Can the upper bound be improved? 
\item  {\it Minimizing fold-curve complexity.}  What is the minimal number of cuspidal points achievable in an unlink formed from the fold curves within a given surface isotopy class?
\item  {\it Global geometric obstructions.}  The $\delta$ invariant provides an obstruction to transforming a single curve into a component of the fold curves.  Can it be extended to determine whether a collection of curves can be made into the full collection of fold curves for a surface?
\item  {\it Higher dimensions.}  Are there similar results for projections of codimension-one submanifolds of $\R^n$ for $n>3$?
\item  {\it Lifting Reidemeister moves.}  Apparent contours can be changed by Reidemeister moves, and by extension of these moves across cusps and moves that add or remove cusps.  When do these correspond to isotopies of a surface in $\R^3$ whose fold curves project to the apparent contours?
\end{enumerate}


\end{document}